\documentclass[10pt]{article}

\usepackage{amsthm}
\usepackage{amsmath}
\usepackage{amssymb}

\def\a{\alpha}
\def\b{\beta}
\def\d{\delta}

\def\e{\varepsilon}
\def\f{\varphi}

\def\Ga{\Gamma}

\def\l{\lambda}
\def\L{\Lambda}
\def\o{\omega}

\def\s{\sigma}

\def\R{{\bf R}}
\def\G{{\bf G}}

\def\Q{{\bf Q}}

\def\Z{{\bf Z}}
\def\C{{\bf C}}
\def\E{{\bf E}}
\def\F{{\bf F}}
\def\P{{\bf P}}

\def\wh{\widehat}
\def\wt{\widetilde}
\def\NN{\mathcal{N}}
\def\FF{\mathcal{F}}
\def\GG{\mathcal{G}}
\def\HH{\mathcal{H}}
\def\LL{\mathcal{L}}
\def\CC{\mathcal{C}}
\def\AA{\mathcal{A}}
\def\BB{\mathcal{B}}
\def\OO{\mathcal{O}}

\def\fg{\mathfrak{g}}
\def\fh{\mathfrak{h}}

\def\fsl{\mathfrak{sl}}
\def\sm{\backslash}
\def\supp{{\rm supp}\,}

\def\be{\begin{equation}}
\def\ee{\end{equation}}
\def\bea{\begin{eqnarray}}
\def\eea{\end{eqnarray}}
\def\bean{\begin{eqnarray*}}
\def\eean{\end{eqnarray*}}

\newtheorem{lem}{Lemma}
\newtheorem{thm}[lem]{Theorem}
\newtheorem{prp}[lem]{Proposition}

\newtheorem{cor}[lem]{Corollary}

\theoremstyle{definition}

\def\sep{\;\vrule\;}

\def\la{\lesssim}
\newcommand{\restr}[1]{\vrule_{\:#1}}

\title{Expansion in $SL_d(\OO_K/I)$, $I$ square-free}
\author{P\'eter P. Varj\'u}
\def\Tr{{\rm Tr}}

\newcommand{\mat}[4]
{\left(
\begin{array}{cc}
#1 & #2 \\
#3 & #4
\end{array}
\right)}

\begin{document}

\maketitle
\begin{abstract}
Let $S$ be a fixed symmetric finite subset of $SL_d(\OO_K)$
that generates a Zariski dense subgroup of $SL_d(\OO_K)$
when we consider it
as an algebraic group over $\Q$ by restriction of scalars.
We prove that the Cayley graphs of $SL_d(\OO_K/I)$ with
respect to the projections of $S$ is an expander family
if $I$ ranges over square-free ideals of $\OO_K$ if $d=2$
and $K$ is an arbitrary numberfield, or if $d=3$ and $K=\Q$.
\end{abstract}

\section{Introduction}
\label{sec_intro}
Let $\GG$ be a graph, and for a set of vertices
$X\subset V(\GG)$, denote by $\partial X$ the set of
edges that connect a vertex in $X$ to one in $V(\GG)\sm X$.
Define
\[
c(\GG)=\min_{X\subset V(\GG),\;\;|X|\le|V(\GG)|/2}
\frac{|\partial X|}{|X|},
\]
where $|X|$ denotes the cardinality of the set $X$.
A family of graphs is called a family of expanders, if
$c(\GG)$ is bounded away from zero for
graphs $\GG$ that belong to the family.
Expanders have a wide range of applications in computer
science (see e.g. Hoory, Linial and Widgerson \cite{HLW} for a recent
survey of expanders and their applications) and
recently they found remarkable applications in pure mathematics
as well (see Bourgain, Gamburd and Sarnak \cite{BGS} and
Long, Lubotzky and Reid \cite{LLR}).

Let $G$ be a group and let $S \subset G$ be a symmetric (i.e. closed
for taking inverses) set of generators.
The Cayley graph $\GG(G,S)$ of $G$ with respect to the
generating set $S$ is defined to be the graph whose vertex set
is $G$, and in which two vertices $x,y\in G$ are connected
exactly if $y\in Sx$.
Let $K$ be a number-field and denote by $\OO_K$ its ring of
integers.
Let $I\subset\OO_K$ be an ideal, and denote by $\pi_I$
the projection $\OO_K\to\OO_K/I$.
In this paper we study the problem whether the graphs
$\GG(SL_d(\OO_K/I),\pi_I(S))$ form an expander family,
where $S\subset SL_d(\OO_K)$ is a fixed symmetric set
of matrices and $I$ runs through certain ideals of $\OO_K$.
This problem was addressed by Bourgain and Gamburd
in a series of papers \cite{BG1}--\cite{BG3}, and by them
jointly with Sarnak in \cite{BGS}.
It is solved for $K=\Q$ in the following cases:
in $\cite{BG1}$ for $d=2$, when
and $I=(p)$ runs through primes, in $\cite{BGS}$ for  $d=2$
and $I=(q)$, $q$ is square-free and in $\cite{BG2}$
and $\cite{BG3}$ when $I=(p^n)$, $p^n$ is a primepower.
(When $d\ge 3$, the prime $p$ has to be kept fixed.)
The necessary and sufficient condition in each case for the Cayley
graphs to be expanders is that $S$ generates a Zariski dense
subgroup $\Ga<SL_d(\C)$.
In $\cite{BGS}$ the expander property is used for $K=\Q(\sqrt{-1})$
for sieving in the context of integral Apollonian packings,
this is our main motivation for extending the problem for
general number-fields.

The starting point for our study is the work of Helfgott
\cite{Hel}, \cite{He2}.
He studies the following problem:
Let $\FF$ be a family of finite fields and let $d\ge2$ be an integer.
Is there a constant $\d>0$ such that for any generating set
$A\subset SL_d(F)$, $F\in\FF$ we have
\be
\label{eq_helfgott}
|A.A.A|\ge|A|\min(|A|,|SL_d(F)|/|A|)^{\d}?
\ee
Here and everywhere in what follows, we use the notation
\[
A.B=\{gh\sep g\in A, h\in B\},
\]
if $A$ and $B$ are subsets of a multiplicative group.
Helfgott answers this question to the affirmative, when
$\FF$ is the family of prime fields and $d=2$ \cite{Hel}
or $d=3$ \cite{He2}.
In section \ref{sec_a4} we show that \cite{Hel}
(i.e. the proof for the case $d=2$) easily extends to the
case of arbitrary finite fields.

Let $r$ be the degree of the number-field $K$, and denote by
$\s_1,\ldots,\s_r$ the embeddings of $K$ into $\C$.
Denote by $\wh \s=\s_1\oplus\cdots\oplus\s_r$ the obvious map
$K\to \C^r$.
This gives rise to an embedding (which will also be denoted by $\wh \s$)
of $SL_d(\OO_K)$ into the direct
product $SL_d(\C)^r$.
Our main result is
\begin{thm}
\label{thm_main}
Let $S\subset SL_d(\OO_K)$ be symmetric and assume that it
generates a subgroup $\Ga<SL_d(\OO_K)$ such that
$\wh \s(\Ga)\subset SL_d(\C)^r$ is Zariski dense.
Assume further that (\ref{eq_helfgott}) holds for some
constant $\d>0$ if $F$ ranges over the fields $\OO_K/P$,
where $P\subset\OO_K$ is a prime ideal.
Then there is an ideal $J\subset\OO_K$ such that
$\GG(SL_d(\OO_K/I),\pi_I(S))$ is a family
of expanders if $I\subset\OO_K$ ranges over square-free ideals prime to $J$.
\end{thm}

It is contained in the claim that $\pi_I(S)$ generates $SL_d(\OO_K/I)$
if $I$ is prime to $J$.
In fact,
$J$ can be taken to be the product of prime ideals $P$ for which
$\pi_P(S)$ does not generate $SL_d(\OO_K/P)$, this fact will be proven
together with the theorem.
We remark that the condition on Zariski density is necessary, otherwise
$\pi_{(q)}(S)$ would not generate $SL_d(\OO_K/(q))$ for any
rational integer q. 
Note that by the above remarks on Helfgott's work, the theorem
is unconditional for $d=2$ and arbitrary $K$ or for $d=3$ and
$K=\Q$.

We introduce some notation that will be used throughout the
paper.
We use Vinogradov's notation $x\ll y$ as a shorthand for $|x|<C y$
with some constant $C$.
Let $G$ be a discrete group.
The unit element of any multiplicatively 
written group is denoted by 1.
For given subsets $A$ and $B$, we denote their product-set by
\[
A.B=\{gh\sep g\in A,h\in B\},
\]
while the $k$-fold iterated product-set of $A$ is denoted by
$\prod_k A$.
We write $\wt A$ for the set of inverses of all elements of $A$.
We say that $A$ is symmetric if $A=\wt A$.
The number of elements of a set $A$ is denoted by $|A|$.
The index of a subgroup $H$ of $G$ is denoted by
$[G:H]$ and we write $H_1\la_L H_2$ if $[H_1:H_1\cap H_2]\le L$
for some subgroups $H_1,H_2<G$.
Occasionally (especially when a ring structure is present) we
write groups additively, then we write
\[
A+B=\{g+h\sep g\in A,h\in B\}
\]
for the sum-set of $A$ and $B$, $\sum_k A$ for the $k$-fold
iterated sum-set of $A$ and 0 for the unit element.

If $\mu$ and $\nu$ are complex valued functions on $G$, we define
their convolution by
\[
(\mu*\nu)(g)=\sum_{h\in G}\mu(gh^{-1})\nu(h),
\]
and we define $\wt \mu$ by the formula
\[
\wt\mu(g)=\mu(g^{-1}).
\]
We write $\mu^{(k)}$ for the $k$-fold convolution of $\mu$ with itself.
As measures and functions are essentially the same on discrete
sets, we use these notions interchangeably, we will also
use the notation
\[
\mu(A)=\sum_{g\in A} \mu(g).
\]
A probability measure is a nonnegative measure with total mass 1.
Finally, the normalized counting measure on a finite
set $A$ is the probability
measure
\[
\chi_A(B)=\frac{|A\cap B|}{|A|}.
\]

We use the same approach to prove Theorem \ref{thm_main}
as in \cite{BG1}--\cite{BGS} which goes back to
\cite{SaX}, we outline this here only,
the details will be given in section \ref{sec_proof}.
Let $\GG$ be an $m$-regular graph, i.e. each vertex is of degree $m$.
It is easy to see that the largest eigenvalue of the adjacency
matrix of $\GG$ is $m$, and it is a simple eigenvalue if
and only if the graph is connected.
Denote by $\l_2(\GG)$ the second largest eigenvalue of the adjacency matrix.
It was proven by Dodziuk \cite{Dod}, Alon and Milman \cite{AlMi}
and Alon \cite{Alo} that a family of graphs is an expander family,
if and only if $m-\l_2(\GG)$ is bounded away from zero, see also
\cite[Theorem 2.4]{HLW}.
For a Cayley graph $\GG(G,S)$, the adjacency matrix is a constant
multiple of convolution by $\chi_S$ from the left considered
as an operator.
Then the multiplicities of the nontrivial eigenvalues are at least
the minimum dimension of a nontrivial representation of $G$.
In the case of $SL_d$ good bounds are know, hence it is enough to estimate
the trace of the operator.
More precisely, with the notation of Theorem \ref{thm_main},
we need to show that for any $\e>0$ there is a constant $C=C(\e,S)$
such that
\be
\label{eq_goal}
\|\pi_I[\chi_S^{(C\log N(I))}]\|_2<|SL_d(\OO_K/I)|^{-1/2+\e},
\ee
where $N(I)$ is the norm of the ideal.
In fact, (\ref{eq_goal}) means that the random walk on
$\GG(SL_d(\OO_K/I),\pi_I(S))$
is close to equidistribution after $C\log N(I)$ steps.

The proof of (\ref{eq_goal}) has two parts, the first is
\begin{thm}
\label{thm_escp}
Let $S\subset SL_d(\OO_K)$ be symmetric, and denote by $\Ga$ the subgroup
it generates.
Assume that $\wh\s(\Ga)$ is Zariski dense in $SL_d(\C)^r$.
Then there is a constant $\d$ depending only on $S$,
and there is a symmetric set $S'\subset \Ga$
such that the following holds.
For any square-free ideal $I$, for any proper
subgroup $H<SL_d(\OO_K/I)$ and for any even integer $l\ge\log N(I)$,
we have
\[
\pi_I[\chi_{S'}^{(l)}](H)\ll[SL_d(\OO_K/I):H]^{-\d}.
\]
\end{thm}
If we know that $g\in\prod_{c\log N(I)}S$, where $c$ is a small constant
depending on $S$, then $\pi_I(g)$ determines $g$ uniquely.
In section \ref{sec_escp}, using Nori's \cite{Nor}
results we give a geometric description of the elements of
$\prod_{c\log N(I)}S$ whose projection modulo $I$ belong to $H$,
this will be a certain subgroup of $SL_d(\OO_K)$.
Then we will prove that the probability for the random walk on
$\GG(\Ga,S)$ to be in this subgroup decays exponentially in the number
of steps we take.
(Actually, first we need to replace $S$ by another set
$S'\subset\Ga$.)
The proof of this is based on a ping-pong argument.

The second part of the proof begins with the following observation.
If we apply Theorem \ref{thm_escp} for $H=\{1\}$, then we already
get
\be
\label{eq_almostthere}
\|\pi_I[\chi_{S'}^{(\log N(I))}]\|\ll \|SL_d(\OO_K/I)\|^{-\d/2}.
\ee
Now working on the quotient $SL_d(\OO_K/I)$,
we can improve on (\ref{eq_almostthere}), if
we take the convolution of 
$\pi_{I}[\chi_{S'}^{(\log N(I))}]$ with itself.
More precisely we prove in section \ref{sec_prod} the following 
\begin{thm}
\label{thm_prod}
Let $G$ be a group satisfying the assumptions (A0)--(A5)
listed in section \ref{sec_prod}.
Then for any $\e>0$, there is some $\d>0$ depending only on
$\e$ and the constants appearing in assumptions (A0)--(A5)
such that the following holds.
If $\mu$ and $\nu$ are probability measures on $G$ such that
$\mu$ satisfies
\[
\|\mu\|_2>|G|^{-1/2+\e}\quad{\rm and}\quad
\mu(gH)<[G:H]^{-\e}
\]
for any $g\in G$ and for any proper subgroup $H<G$, then
\[
\|\mu*\nu\|_2<\|\mu\|_2^{1/2+\d}\|\nu\|_2^{1/2}.
\]
\end{thm}
Assumptions (A0)--(A5) are too technical so we do not list
them here in the introduction.
Among other things, we assume that $G$ is the direct product of such 
quasi-simple
groups that also satisfy the conclusion of the theorem.
To prove the latter for the groups $SL_d(\OO_K/P)$ we need
(\ref{eq_helfgott}), and this is the reason why we have Theorem
\ref{thm_main} only in the cases, when (\ref{eq_helfgott}) is available.
The quasi-simplicity of the factors is a severe restriction, for example
it excludes factors of the form $SL_d(\OO_K/P^k)$, where $P$ is a
prime ideal.
Therefore a new idea is needed to prove Theorem \ref{thm_main} for
general ideals.

A similar result for $G=SL_2(\Z/q\Z)$, $q$ square-free
(under stronger hypothesis on $\mu$)
is given by Bourgain, Gamburd and Sarnak \cite[Proposition 4.3]{BGS}.
They use an argument similar to Helfgott's \cite{Hel} to reduce it
to a so-called sum-product theorem for the ring $\Z/q\Z$.
Then they prove the latter by reducing it to the case of $\Z/p\Z$,
$p$ prime.
The difference in our approach is that we use Helfgott's theorem
as a black box, and extend it to the case of square-free modulus
in a way that very much resembles the proof given in
\cite[section 5]{BGS}
for the sum-product theorem.

\bigskip
{\bf Acknowledgement.}
I am very grateful to my advisor, Jean Bourgain for suggesting this
problem and for guiding me during my research.
I also had very useful discussions with Elon Lindenstrauss, Alireza Salehi Golsefidy
and Peter Sarnak, I thank them for their interest and for their valuable
remarks.

While writing this paper, I was supported by a
Fulbright Science and Technology Award, grant no. 15073240.

\section{Escape of mass from subgroups}
\label{sec_escp}

We prove Theorem \ref{thm_escp} in this section.
First we note that we may assume that $I$ is a principal
ideal generated by a square-free rational integer $q$.
Indeed, there is always a square-free rational integer $q\in I$
such that $q\le N(I)$.
Let $\wh H$ be the preimage of $H$ under the projection
$SL_d(\OO_K/(q))\to SL_d(\OO_K/I)$.
Then we have $\log N((q))\ge\log N(I)\ge \log N((q))/r$ and
$[SL_d(\OO_K/I):H]=[SL_d(\OO_K/(q)):\wh H]$.
Hence
the claim of the theorem for $I$ and $H$ follows from
the claim for $(q)$ and $\wh H$.
In what follows we assume that $I=(q)$ and write $\pi_q=\pi_{(q)}$.
Let $q=p_1\ldots p_n$ be the prime factorization of $q$ and assume
without loss of generality that none of the $p_i$ ramify in $K$.

For $g\in SL_d(\C)$ denote by $\|g\|$ the operator norm of $g$
with respect to the $l^2$ norm on $\C^d$.
If $\|g\|<\sqrt{q}/2$ for some $g\in SL_d(\OO_K)$, then
clearly $\|g'\|>\sqrt{q}/2$ for any other $g'\in SL_d(\OO_K)$
with $\pi_q(g)=\pi_q(g')$ since $\|g\|\ge\sqrt{q}$ if $\pi_q(g)=0$
and $g\neq0$.
Hence elements of small norm are determined uniquely by their
projections modulo $q$.
The first step towards the proof of Theorem \ref{thm_escp}
is to study when the projection of an element of small norm
belong to $H$, i.e. we study the set
\[
\LL_\d(H):=\{h\in SL_d(\OO_K)\sep\pi_q(h)\in H,\;
\|\wh\s(h)\|<[SL_d(\OO_K/(q)):H]^\d\}
\]
for $\d>0$ and for $H<SL_d(\OO_K/(q))$.

By Weil restriction of scalars, we consider $SL_d(K)$ as
the $\Q$--points of an algebraic group.
To fix notation, we describe this process in detail.
Let $e_1,\ldots,e_r$ be an integral basis of $\OO_K$.
Multiplication by an element $a\in K$ is an endomorphism
of the $\Q$--vectorspace $K$.
This gives rise to an embedding $\a:K\to Mat_r(\Q)$ onto
a subalgebra of $Mat_r(\Q)$ which is defined by linear
equations over $\Q$.
Thus there is an algebraic subgroup $\G$ of $SL_{dr}$
defined over $\Q$ such that $SL_d(K)$ is isomorphic
to $\G(\Q)$ as an abstract group, we denote this isomorphism by $\a$ as well.
Moreover, we have $\a(SL_d(\OO_K))=\G(\Q)\cap SL_{dr}(\Z)$.
To shorten notation, we write $\G(\Z)=\G(\Q)\cap SL_{dr}(\Z)$.
The image of $e_1,\ldots, e_r$ under $\pi_q$
is a basis of the $\Z/q\Z$--module $\OO_K/(q)$, hence
$\a$ induces an isomorphism from $SL_d(\OO_K/(q))$
to $\G(\Z/q\Z)$.
Denote by $\fg$ the Lie-algebra of $\G$, then $\fg(\Q)$
is a subspace of $Mat_{dr}(\Q)$ defined by (linear)
polynomials $\f_1,\ldots,\f_{d^2r^2-r(d^2-1)}\in\Z[x]$.
If $p$ is a prime which does not ramify in $K$, then
we can write $(p)=P_1\ldots P_k$ with different prime
ideals $P_i$.
Then $\G(\Z/p\Z)$ is isomorphic to
$SL_d(\OO_K/P_1)\times\cdots\times SL_d(\OO_K/P_k)$.

\begin{prp}
\label{prp_alg}
There are constants $C$ and $\d$ depending only on $K$
such that the following holds.
For any subgroup $H<SL_d(\OO_K/(q))$, there are
$u,v\in\fg(\C)$ and there is a subgroup $H^{\sharp}<H$
with $[H:H^{\sharp}]<C^n$,
such that if $h\in\LL_\d(H^\sharp)$
then
\[
\Tr(\a(h)u\a(h)^{-1}v)=0,
\]
but there is some $g_0\in\G(\Q)$ such that $\Tr(g_0ug_0^{-1}v)=1$.
\end{prp}
In what follows we often write $\F_{p^m}$
for the finite field of order $p^m$.
Recall that $n$ is the number of prime factors of $q$ and
$q=p_1\cdots p_n$.
Then $\G(\Z/q\Z)=\G(\F_{p_1})\times\ldots\times\G(\F_{p_n})$.
For $q_1|q$, denote by
\[
\pi_{q_1}:\G(\Z/q\Z)\to\times_{p|q_1}\G(\F_p)
\]
the projection
to the product of direct factors corresponding to the
prime factors of $q_1$.
Fix a proper subgroup $H<\G(\Z/q\Z)$ and
denote by $q_1$ the product of all primes $p|q$ for which
$\pi_{p}(H)=\G(\F_p)$.
In the course of the proof we will replace $q$ by $q/q_1$
and $H$ by $\pi_{q/q_1}(H)$.
We need to show that $[\G(\Z/(q/q_1)\Z):\pi_{q/q_1}(H)]$
is not much smaller than $[\G(\Z/q\Z):H]$.
For this we first give
\begin{lem}
\label{lem_univalg}
Let $p_1$ and $p_2$ be two different primes and assume
that $N\lhd H<SL_{d}(\F_{p_2^{m_1}})$ such that
$H/N$ is isomorphic to $PSL_{d}(\F_{p_1^{m_2}})$
with some integers $m_1$, $m_2$.
Then
\[
p_1|\prod_{i=2}^{d}(p_2^{im_1}-1),
\]
in particular, for a fixed $p_2$ the product of all primes, which
can arise as $p_1$, is at most $p_2^{d^2m_1}$.
\end{lem}
\begin{proof}
As $PSL_{d}(\F_{p_1^{m_2}})$ has an element of order $p_1$ and
the order of $SL_{d}(\F_{p_2^{m_1}})$ is
$p_2^{m_1d(d-1)/2}\prod_{i=2}^{d}(p_2^{im_1}-1)$,
the assertion is clear.
\end{proof}

\begin{lem}
\label{lem_subprod}
Let $H$ be a subgroup of $G=\G(\Z/q\Z)$ and denote
by $q_1$ the product of primes $p|q$ with $\pi_p(H)=\G(\F_p)$
and set $q_2=q/q_1$.
There is a subgroup $H_2<\G(\Z/q_2\Z)$ of the form
$\times_{p|q_2}H_p$, where each $H_p$ is a proper
subgroup of $\G(\F_p)$ such that $\pi_{q_2}(H)<H_2$ and
\[
[\G(\Z/q_2\Z):H_2]>[G:H]^c
\]
with a constant $c$ depending only on $d$ and $r$.
\end{lem}
\begin{proof}
If for some $p|q_1$, $\G(\F_p)$ is a direct factor of $H$
then
\[
[\G(\Z/(q/p)\Z):\pi_{q/p}(H)]=[G:H],
\]
hence we can assume without loss of generality
that there is no such prime.
We show that for each $p_1|q_1$, there is some $p_2|q_2$
such that the conditions of the previous lemma are satisfied.
This will yield a bound on $q_1$.
Set $q'=q/p_1$
By Goursat's Lemma, there is a nontrivial group $N$ and surjective
homomorphisms
\[
\f:\pi_p(H)=\G(\F_{p_1})\to N, \qquad
\psi:\pi_{q'}(H)\to N.
\]
For each factor $p|q'$, $\psi$ gives rise to a surjective homomorphism
\[
\psi_p:\pi_p(H)\to N_p=N/\{\psi(h)\sep h\in \pi_{q'}(H),\;\pi_p(h)=1\}
\]
in the obvious way.
Since the intersection of all the subgroups
$\{\psi(h)\sep h\in \pi_{q'}(H),\;\pi_p(h)=1\}$
is trivial, there is a prime $p_2$ for which $N_{p_2}$
is nontrivial.
As $\G(\F_{p_1})$ and $\G(\F_{p_2})$ has no nontrivial common
factors, $p_2|q_2$.
%
It is clear that $p_1$ and $p_2$ satisfy the conditions of Lemma
\ref{lem_univalg}, whence $q_1<q_2^{rd^2}$.

For each $p|q_2$ let $H_p$ be a proper subgroup
of $\G(\F_p)$ containing $\pi_p(H)$.
Since $\G(\F_p)$ is generated by its subgroups isomorphic to
$SL_2(\F_p)$, there must be at least one such subgroup
which is not contained in $H_p$.
Any proper subgroup of $SL_2(\F_p)$ is of index at least $p+1$,
hence $[\G(\F_p):H_p]>p$.
This shows that for $H_2=\times_{p|q_2}H_p$, we have
\[
[\G(\Z/q_2\Z):H_2]>q_2> q^{1/(d^2r+1)}>[G:H]^c.
\]
\end{proof}

The proof of Proposition \ref{prp_alg} is based on the description of
subgroups of $GL_{d}(\F_p)$ given by Nori \cite{Nor}
that we recall now.
Let $H$ be a subgroup of $GL_{d}(\F_p)$
and denote by
$H^+$ the subgroup of $H$ generated by its elements of order $p$.
\cite[Theorem B]{Nor} states that if $p$ is
bigger than a constant depending only on $d$, then there is
a connected algebraic subgroup $\wt H$ of $GL_d$ defined
over $\F_p$ such that $H^+=\wt H(\F_p)^+$.
Denote by $\fh$ the Lie algebra of $\wt H$, and
define $\exp$ and $\log$ by
\[
\exp(z)=\sum_{i=0}^{p-1}\frac{z^i}{i!}\qquad{\rm and}\qquad
\log(z)=-\sum_{i=1}^{p-1}\frac{(1-z)^i}{i}
\]
for $z\in Mat_{d}(\F_p)$.
Then for $p$ large enough,
$\exp$ and $\log$ sets up a one to one correspondence
between elements of order $p$ of $H^+$ and nilpotent elements
$\fh(\F_p)$ by \cite[Theorem A]{Nor}.
Moreover $\fh(\F_p)$ is spanned by its nilpotent elements.
To understand subgroups not generated by the elements
of order $p$, we will use \cite[Theorem C]{Nor}
which asserts that if $p\ge d$, then
there is a commutative subgroup $F<H$ such that
$FH^+$ is a normal subgroup of $H$ and its index $[H:FH^+]$
is bounded in terms of $d$.

\begin{proof}[Proof of Proposition \ref{prp_alg}]
We follow the argument in \cite[Proposition 4.1]{BG3}.
Recall that $H$ is a subgroup of $SL_d(\OO_K/(q))$.
Apply Lemma \ref{lem_subprod} to $\a(H)$ to get a modulus $q_2|q$
and a subgroup $H_2<\G(\Z/q_2\Z)$.
Suppose that the proposition holds for $\a^{-1}(H_2)$ and for an
$H_2^\sharp<SL_d(\OO_K/(q_2))$ with
$[H_2:\a(H_2^\sharp)]<C^n$.
Set
\[
H^\sharp=\{h\in H\sep \pi_{q_2}(h)\in H_2^\sharp\},
\]
and observe that $[H:H^\sharp]<C^n$ and
$\LL_\d(H^\sharp)\subset\LL_{\d/c}(H_2^{\sharp})$ with the constant
$c$ from Lemma \ref{lem_subprod}.
Therefore, if the proposition holds for $\a^{-1}(H_2)$ and
$H_2^{\sharp}$, it also holds for $H$ and $H^\sharp$.
We assume in what follows that $\a(H)=H_{p_1}\times\ldots\times H_{p_n}$,
where $q=p_1\cdots p_n$ is the prime factorization of $q$
and $H_{p_i}$ is a proper subgroup of $\G(\F_{p_i})$.
For each direct factor $H_{p_i}$, let $H^{\sharp}_{p_i}<H_{p_i}$
be such that $H^{\sharp}_{p_i}/H^+_{p_i}$ is commutative and
$[H_{p_i}:H^{\sharp}_{p_i}]<C$ with a constant $C$
depending on $r$ and $d$, such a subgroup exists by
\cite[Theorem C]{Nor}.
Define $H^\sharp=\a^{-1}(H_{p_1}^{\sharp}\times
\ldots\times H_{p_n}^{\sharp})$.

For each $g\in\G(\Z)$ define the polynomial $\eta_g\in\Z[X,Y]$
with $X=(X_{l,k})_{1\le l,k\le dr}$ and
$Y=(Y_{l,k})_{1\le l,k\le dr}$ by
\[
\eta_g(X,Y)=\Tr(gXg^{-1}Y).
\]
Let $A$ be a fixed set of
generators of $\G(\Z)$ and
fix an element $g_0\in A$.
Consider
the system of equations
\be
\label{eq_poly}
\begin{split}
\f_i(X)&=0\qquad 1\le i\le r^2d^2-r(d^2-1),\\
\f_i(Y)&=0\qquad 1\le i\le r^2d^2-r(d^2-1),\\
\eta_{\a(h)}(X,Y)&=0\qquad{\rm for}\; h\in\LL_\d(H^\sharp),\\
\eta_{g_0}(X,Y)&=1,
\end{split}
\ee
where $\d$ is a small constant depending on $d$ and $r$
to be chosen later.
Recall that $\f_i$ are the polynomials defining the Lie algebra
$\fg$.
The assertion follows once we show that (\ref{eq_poly})
has a solution $X=u,Y=v\in Mat_{rd}(\C)$ for an appropriate
choice of $g_0$.

First we show that for each $p=p_i$, there is at least one
$g_0\in A$ such that (\ref{eq_poly}) has
a solution in $Mat_{dr}(\F_p)$.
We apply the results of \cite{Nor} for $H=H_p$, in particular
let $\wt H$ and $\fh$ be the same
as in the discussion preceding the proof.
Conjugation by an element
$g\in\G(\F_p)$ permutes elements of order $p$ of $H_p^+$
if and only if it permutes nilpotent elements of $\fh(\F_p)$.
Hence $\fh(\F_p)$ is invariant under $g$ in the adjoint representation,
exactly if $g$ is in the normalizer of $H_p^+$.
First we consider the case when $H_p^+$ is not a normal
subgroup of $\G(\F_p)$.
Then there is at least one element $\pi_p(g_0)\in \pi_p(A)$ whose adjoint
action does not leave $\fh(\F_p)$ invariant. 
Let $u\in\fh(\F_p)$ be such that
$\pi_p(g_0)u\pi_p(g_0)^{-1}\notin\fh(\F_p)$
and
let $v\in\fg(\F_p)$ be orthogonal to $\fh(\F_p)$ with respect to
the non-degenerate bilinear form $\langle x,y\rangle=Tr(xy)$
and such that $\Tr(\pi_p(g_0)u\pi_p(g_0^{-1})v)=1$.
This settles the claim.
Now consider the case when $H_p^+\lhd\G(\F_p)$.
If $(p)=P_1\ldots P_k$ is the factorization of $(p)$ over $K$, then
$\G(\F_p)$ is isomorphic to
$SL_d(\OO_K/P_1)\times\cdots\times SL_d(\OO_K/P_k)$,
and $H_p^+$ must be the direct
product of some of these factors.
Consider a direct factor $SL_d(\OO_K/P_i)$ which do not appear
in $H_p^+$ and denote by $N$ the projection of $H_p^\sharp$
to this factor.
There is a Lie subalgebra $\fg_i(\F_p)\subset\fg(\F_p)$ which is isomorphic
to $\fsl_d(\OO_K/P_i)$, invariant and irreducible in the adjoint
representation of $\G(\F_p)$ and the adjoint action of an element
$g\in \G(\F_p)$ on $\fg_i(\F_p )$ is determined by its
projection to the factor $SL_d(\OO_K/P_i)$.
If $N$ is nontrivial denote by $V$ the intersection of the
$\OO_K/P_i$-linear span of $N$ in $Mat_{d}(\OO_K/P_i)$
and the lie algebra $\fg_i(\F_p)$.
If $N$ is trivial, let $V$ be any proper subspace
of $\fg_i(\F_p)$.
Then $V$ is again invariant under $H_p^{\sharp}$ in the adjoint
representation but not under $\G(\F_p)$ and
we can establish the claim the same way as above.

For a particular $g_0\in A$,
denote by $q_{g_0}$ the product of primes $p|q$
for which (\ref{eq_poly}) has a solution over $\F_p$.
As there are only a finite number (and bounded
in terms of $K$) of possibilities for $g_0$,
there is an appropriate choice such that $q_{g_0}>q^c$.
Here and everywhere below $c$ is a constant depending only on $K$
which need not be the same at different occurrences.
Now assume to the contrary that the system (\ref{eq_poly}) has no
solution over $\C$.
We can clearly replace the family of polynomials $\eta_\a(h)$,
$h\in\LL_\d(H^\sharp)$ by a linearly independent
subset of at most $M\le r^4d^4$ elements that we denote
by $\eta_{1},\ldots,\eta_M$.
Note that the coefficients of all the polynomials in
(\ref{eq_poly}) are bounded by $[G:H]^{c\d}<q^{c'\d}$.
Using the effective Bezout identities proved by Berenstein
and Yger \cite[Theorem 5.1]{BeYg} we obtain polynomials
\bean
\psi_1(X,Y),\ldots,\psi_M(X,Y)\in\Z[X,Y],\\
\psi'_1(X,Y),\ldots,\psi'_{r^2d^2-r(d^2-1)}(X,Y)\in\Z[X,Y],\\
\psi''_1(X,Y),\ldots,\psi''_{r^2d^2-r(d^2-1)}(X,Y)\in\Z[X,Y],\\
\psi'''(X,Y)\in\Z[X,Y]
\eean
and a positive integer $0<D<q^{c\d}$
such that
\bean
D&=&\sum_{i=1}^M\eta_i(X,Y)\psi_i(X,Y)\\
&&+\sum_{i=1}^{r^2d^2-r(d^2-1)}\f_i(X)\psi'_i(X,Y)\\
&&+\sum_{i=1}^{r^2d^2-r(d^2-1)}\f_i(Y)\psi''_i(X,Y)\\
&&+(\eta_{g_0}(X,Y)-1)\psi'''(X,Y).
\eean
Substituting the solution of (\ref{eq_poly}) over $\F_p$
for all $p|q_{g_0}$, we see that $q_{g_0}|D$, a contradiction
if $\d$ is small enough.
\end{proof}

\begin{cor}
\label{cor_alg}
There are constants $\d$ and $C$ depending only on $K$, and for
each $H<SL_d(\OO_K/(q))$ there is an $H^\sharp<H$
with $[H:H^\sharp]<C^n$ such that
at least one of the following holds:
\begin{enumerate}
\item
There is an embedding $\s: K\to\C$ and a proper subspace
$V\subset\fsl_d(\C)$ such that if $h\in\LL_\d(H^\sharp)$, then
\be
\label{eq_var1}
\s(h)V\s(h^{-1})=V.
\ee
\item
There are two embeddings $\s_1,\s_2:K\to\C$ and an invertible linear
transformation $T:\fsl_d(\C)\to\fsl_d(\C)$ such that
\be
\label{eq_var2}
T(\s_1(h)v\s_1(h^{-1}))=\s_2(h)T(v)\s_2(h^{-1})
\ee
for any $h\in\LL_\d(H^\sharp)$ and $v\in\fsl_d(\C)$.
\end{enumerate}
\end{cor}
\begin{proof}
Choose $\d$ to be $1/r(d^2-1)$ times the $\d$ in Proposition
\ref{prp_alg}.
Then there are $u,v\in\fg(\C)$ and there is a $g_0\in\G(\Q)$
such that $\Tr(\a(h)u\a(h^{-1})v)=0$ for
\[
h\in\textstyle\prod_{r(d^2-1)}
\LL_\d(H^{\sharp})\subset\LL_{\d r(d^2-1)}(H^\sharp),
\]
while $\Tr(g_0ug_0^{-1}v)=1$.
Let $U_l$ be the linear span of
$\{\a(g)u\a(g^{-1})\sep g\in\prod_{l}\LL_\d(H^\sharp)\}$ in $\fg(\C)$.
Comparing dimensions, we see that for some
$l\le r(d^2-1)$ we have $U_l=U_{l+1}$, and then it
is invariant under $\a(\LL_\d(H^\sharp))$ in the adjoint representation.
Write $U=U_l$.
Then for any $x\in U$, we have $\Tr(xv)=0$, hence
$g_0ug_0^{-1}\notin U$, and $U$ is not invariant under the full
group $\G(\C)$ in the adjoint representation.

Consider the embedding $\a:K\to Mat_r(\Q)$.
Let $a\in K$ be a generator of $K$ over $\Q$.
Note that the minimal polynomial of $a$ over $\Q$
is the same as the minimal polynomial of $\a(a)$
in $Mat_r(\Q)$.
This polynomial has $r$ different roots
$\s_1(a),\ldots,\s_r(a)$ in $\C$, hence there is a
basis over $\C$ in which $\a(a)$ is diagonal.
Any element $b\in K$ can be expressed as the value at $a$
of a polynomial with rational coefficients.
Thus in that basis the matrix of $b$ is
$diag(\s_1(b),\ldots,\s_l(b))$.
Therefore there is an appropriate basis in which any
$g\in\G(\C)$ is a block diagonal matrix with
$\s_1(g),\ldots,\s_r(g)$ along the diagonal.
This gives rise to an isomorphism $\b:\G(\C)\to SL_d(\C)^r$
such that $\s=\b\circ\a$.
$\b$ also induces an isomorphism between the lie algebras
$\fg(\C)$ and $\fsl_d(\C)^r$, denote by $W$ the image of $U$.

Assume that $W$ is a subspace of minimal dimension which is invariant
under $\wh\s[\LL_\d(H^\sharp)]$ in the adjoint representation,
but not under the whole group $SL_d(\C)^r$.
Denote by $\fg_1(\C),\ldots,\fg_r(\C)$ the $r$ copies of $\fsl_d(\C)$
in $\fsl_d(\C)^r$ and denote by $\pi_i$ the projection to $\fg_i(\C)$.
For $1\le i\le r$, the spaces $\pi_i(W)$ and $W\cap\fg_i(\C)$
are invariant under $\s_i[\LL_\d(H^\sharp)]$ in the adjoint representation,
hence 1. holds if the dimension of any of the above spaces
is strictly between 0 and $d^2-1$.
Suppose that this is not the case.
Since $W$ is not the direct sum of some $\fg_{i}(\C)$, we may assume that
say 
$W\cap\fg_1(\C)=\{0\}$ and
$\pi_1(W)=\fg_1(\C)$.
By the minimality of the dimension of $W$, $Ker(\pi_1)\cap W$ must
be the direct sum of some $\fg_i(\C)$.
Since $\dim W>\dim Ker(\pi_1)\cap W$,
we can assume that say $\pi_2(Ker(\pi_1)\cap W)=\{0\}$
and $\pi_2(W)=\fg_2(\C)$.
Then $T=\pi_2\circ\pi_1^{-1}$ is well-defined and satisfies 2.
\end{proof}

Recall that we are given a symmetric $S\subset SL_d(\OO_K)$
which generates the subgroup $\Ga$.
We will choose an appropriate $S'\subset\Ga$ and study the
random walk on $\G(\langle S'\rangle,S')$, where $\langle S'\rangle$
is the subgroup generated by $S'$.
In particular, we prove an exponential decay for the probability
that after $k$ steps we are in the subgroup of $SL_d(\OO_K)$
whose elements satisfy (\ref{eq_var1}) for some fixed $V$
or in the one whose elements satisfy (\ref{eq_var2}) for some fixed $T$.

\begin{prp}
\label{prp_subspc}
Assume that $\wh\s(\Ga)$ is Zariski dense in $SL_d(\C)^r$.
Let $V$ be a proper subspace of $\fsl_d(\C)$, and
let $\s:K\to\C$ be an embedding,
denote by
$H_V$ the subgroup of elements $h\in SL_d(\OO_K)$
for which (\ref{eq_var1}) holds.
Then
\[
\chi_S^{(k)}(H_V)\ll c^k
\]
with some constant $c<1$ depending only on $S$.
\end{prp}

\begin{prp}
\label{prp_trans}
Assume that $\wh\s(\Ga)$ is Zariski dense in $SL_d(\C)^r$.
Then there is a symmetric set $S'\subset\Ga$, and a constant $c<1$
depending only on $S$ such that the following holds.
Let $\s_1,\s_2$ be two different embeddings of $K$ into $\C$ and
let $T$ be an invertible linear transformation on $\fsl_d(\C)$.
Denote by $H_T$ the subgroup of elements $h\in SL_d(\OO_K)$
for which (\ref{eq_var2}) holds.
Then
\[
\chi_{S'}^{(k)}(H_T)\ll c^k.
\]
\end{prp}

Proposition \ref{prp_subspc}
can be proved as it is outlined in \cite[Section 9.]{BG2},
we ommit the details.
A weaker form analogous to Proposition \ref{prp_trans},
which is sufficient for our purposes, can be proved
by the same method as we prove Proposition \ref{prp_trans} below.

Let $A\subset\Ga$ be a subset that freely generates a subgroup.
By abuse of notation, on a word $w$ over $A\cup \wt A$,
we mean a finite sequence $g_1g_2\cdots g_k$, where
$g_1,\ldots,g_k\in A\cup \wt A$.
Recall that $\wt A$ is the set of inverses of all elements of $A$.
We will refer to the elements of $A\cup\wt A$ as letters.
We say that $w$ is reduced if
$g_ig_{i+1}\neq1$ for any $1\le i<k$.
There is a natural bijection between the set of reduced words
and the group $\langle A\rangle$ generated by
$A\subset\Ga$.
For the sake of clarity we write $w_1.w_2$ for concatenation
of the sequences $w_1$ and $w_2$ and $w_1w_2$ for the product in $\Ga$,
i.e. for concatenation followed by all possible reductions.
Denote by $B_l$ the set of reduced words of length $l$.
Note that $|B_l|=2m(2m-1)^{l-1}$ for $l\ge 1$.

\begin{lem}
\label{lem_alg}
Let notation be as above, and suppose that
$H<\langle A\rangle$ is a subgroup such that for any
$h\in\langle A\rangle$, there
is a letter $g_0\in A\cup\wt A$ such that $w\notin hHh^{-1}$ whenever
$w$ is a reduced word starting with $g_0$.
Then we have
\[
|B_l\cap H|\le(2m-1)^{l/2+1}(2m-2)^{l/2-1}.
\]
\end{lem}

We remark that the condition for $h=1$ can be interpreted
as follows.
We can remove one edge incident to $1$ from the Schreier
graph of $H\sm G$ such that we get two connected components and one
of these is a tree.

\begin{proof}
Let $w_0$ be the longest word (possibly the empty word 1)
such that $w_0$ is a prefix of all non-unit elements of $H$.
Let $w_1$ be a reduced word of length at most
$\lceil l/2\rceil-1$.
We want to bound the number of letters $g'\in A\cup\wt A$ that can be
the next letter in a reduced word of length $l$ which belongs to
$H$.
We will show that if $|w_1|>|w_0|$ then there are at most $2l-2$
such letters.
If $|w_1|=|w_0|$, we will see that there are at most $2l-1$ choices
for $g'$, this being trivial if $w_0\neq1$.
If $|w_1|<|w_0|$ then we always have exactly one choice.
Thus if we pick the letters of $w\in S_{l}\cap H$ one by one,
then at the first $\lceil l/2\rceil$ steps we have at most
$2l-2$ choices with possibly one exception, when we might have $2l-1$,
this gives the claim.

Now assume that $|w_0|<|w_1|\le \lceil l/2\rceil-1$,
but if $w_0=1$, we allow
$w_1=1$.
Using the assumption for $h=w_1^{-1}$, we get a letter $g_0$
such that if $g_0.w_2$ is a reduced word (i.e. the first
letter of $w_2$ is not $g_0^{-1}$), then $g_0.w_2\notin w_1^{-1}Hw_1$.
We show that the last letter of $w_1$ is not $g_0^{-1}$.
If $w_1$ is not the empty word, it is longer than $w_0$, hence
there is a word $u\in H$, $w_1$ is not a prefix of which.
Now if $g_0^{-1}$ was the last letter of $w_1$, we would have
$w_1^{-1}uw_1\in w_1^{-1}Hw_1$ which begins with
$g_0$, a contradiction.

Obviously we can not continue $w_1$ with the inverse of its last
letter to get a reduced word.
We show that we  can not continue it with $g_0$ either to get
one in $B_l\cap H$.
Assume to the contrary that for some $w_2$,
$w_1.g_0.w_2$ is a reduced word in $B_l\cap H$.
Then $g_0w_2w_1\in w_1^{-1}Hw_1$ and the length of $w_1$
is less than the length of $w_2$, hence $g_0w_2w_1$ starts with
$g_0$, a contradiction.
\end{proof}

Let $V$ be a vectorspace over $\C$, and denote
by $\P(V)$ the corresponding projective space.
For a vector $v\in V$ (for a subspace $W\subset V$)
denote by $\bar v$ ($\overline{W}$) its projection
to $\P(V)$.
Any invertible linear transformation $T$ of $V$
acts naturally on $\P(V)$, this
action will be denoted by the same letter.
We say that $T$ is proximal, if $V$ is spanned by an eigenvector
$z_T$ and an invariant subspace $V_T$ of $T$ and the eigenvalue
corresponding to $z_T$ is strictly larger than any other eigenvalue
of $T$.
In short, $T$ is proximal if it has a unique simple eigenvalue
of maximal modulus.
It is clear that whenever $z_T$ and $V_T$ exist, $V_T$ is unique and
$z_T$ is unique up to a constant multiple.
Define the distance on $\P(V)$ by
\[
d(\bar x,\bar y)=\frac{\|x\wedge y\|}{\|x\|\|y\|},
\]
where $\|\cdot\|$ is the norm coming from the standard Hermitian form.
We recall from Tits \cite{Tit} a simple criterion for a transformation $T$ to be
proximal.
Let $Q\subset\P(V)$ be compact and assume that $T(Q)$ is contained
in the interior of $Q$.
Assume further that $d(T(x),T(y))<d(x,y)$ for $x,y\in Q$.
Then $T$ is proximal and $\bar z_T\in Q$, see \cite[Lemma 3.8 (ii)]{Tit}.

Let notation be the same as in Proposition \ref{prp_trans}.
For $i\in\{1,2\}$, denote by $\rho_i$ the representation
of $SL_d(\OO_K)$ on $\fsl_d(\C)$ defined by
\[
\rho_i(h)v=\s_i(h)v\s_i(h^{-1})\qquad {\rm for}\quad v\in\fsl_d(\C)\quad
{\rm and}\quad h\in SL_d(\OO_K).
\]
We study the action of $SL_d(\OO_K)$ on the space
$\P(\fsl_d(C))\times\P(\fsl_d(C))$ via $\rho_1\oplus\rho_2$.
If $T$ is an invertible linear transformation of $\fsl_d(\C)$ and $h\in H_T$
is an element such that $\rho_1(h)$ and $\rho_2(h)$ are both proximal,
then
\be
\label{eq_T}
T(\bar z_{\rho_1(h)})=\bar z_{\rho_2(h)}
\ee
clearly.
Our aim is to find a subset $A\subset\Ga$ such that
$A$ freely generates a subgroup of $SL_d(\OO_K)$ and for any
linear transformation $T$ of $\fsl_d(\C)$, there is a letter
$g_0\in A\cup\wt A$ such that (\ref{eq_T}) fails when
$h=w$ is a reduced word starting with $g_0$.
Then Proposition $\ref{prp_trans}$ will follow easily from Lemma
\ref{lem_alg}.

We say that $A\subset SL_d(\OO_K)$ is generic,
if for any $g\in A\cup\wt A$, $\rho_1(g)$ and $\rho_2(g)$
are both proximal, and the following hold:
\begin{itemize}
\item[(i)] for every $g_1,g_2\in A\cup\wt A$ with
$g_1g_2\neq1$ and $i\in\{1,2\}$,
we have $z_{\rho_i(g_1)}\notin V_{\rho_i(g_2)}$,

\item[(ii)] for any proper subspace $V$ of $\fsl_d(\C)$ of dimension
$k$ and $i\in\{1,2\}$, we have
\[
|\{g\in A\cup\wt A\sep z_{\rho_i(g)}\in V\}|\le k+1,
\]

\item[(iii)]
for any linear transformation $T$ on $\fsl_d(\C)$, we have
\[
|\{g\in A\cup\wt A\sep T(\bar z_{\rho_1(g)})=\bar z_{\rho_2(g)}\}|\le d^2+1.
\]
\end{itemize}
Note that $\fsl_d(\C)$ is of dimension $d^2-1$.
Actually the above definition would be more natural if we replaced the
right hand sides of the inequalities in (ii) and (iii) by $k$ and $d^2$
respectively, however doing so would make the next proof slightly more
complicated.
We prove the existence of generic sets in

\begin{lem}
\label{lem_generic}
Assume that $\wh\s(\Ga)$ is Zariski dense in $SL_d(\C)$.
Then for $m$ positive integer, there is a generic set
$A_m\subset\Ga$ of cardinality $m$.
\end{lem}
\begin{proof}
Goldsheid and Margulis \cite{GoMa} proves
(see also sections 3.12--3.14 in Abels, Margulis and Soifert \cite{AMS})
that if a real algebraic subgroup of $GL_d(\R)$ is strongly
irreducible (i.e. does not leave a finite union of proper subspaces
invariant) and contains a proximal element, then a Zariski dense
subgroup of it also contains a proximal element.
If $\s_1$ is a real embedding, then it follows from the Zariski
density of $\s_1(\Ga)$ in $SL_d(\R)$,
that there is an element $g_0\in\Ga$ such that $\s_1(g_0)$ is proximal.
If $\s_1$ is complex, then let $\bar\s_1$ denote its
complex conjugate.
Since $(\s_1\oplus\bar\s_1)(\Ga)$ is
Zariski dense in $SL_d(\C)\times SL_d(\C)$, we get that 
$\s_1(\Ga)$ is Zariski dense in $SL_d(\C)$
over the reals as well, i.e. considered as a subgroup of $SL_{2d}(\R)$.
Consider $\C^d$ as a real vectorspace, and take the wedge product
$\C^d\wedge \C^d$.
Denote by $U$ the subspace spanned by the images of complex
lines in $\C^d$, this is also the subspace fixed by the
linear transformation induced from the transformation
multiplication by $i$ on $\C^d$.
It is clear that $SL_d(\C)$ (as a real group) acts on $U$ strongly
irreducibly and proximally in the natural way, hence there is an element
$g_0\in\Ga$ such that $\s_1(\Ga)$ is proximal on $U$.
This implies in turn that $\s_1(\Ga)$ is proximal on $\C^d$ now considered
as a complex vectorspace.
Denote by $\s_i'$  (for $i\in\{1,2\}$) the representation of $\Ga$
which assigns the transpose inverse of the matrix assigned by $\s_i$.
Applying \cite[Lemma 5.15]{AMS} for the representation
$\s_1\oplus\s_1'\oplus\s_2\oplus\s_2'$, we get an element
$g_0\in\Ga$ such that $\s_1(g_0)$, $\s_1(g_0^{-1})$,
$\s_2(g_0)$ and $\s_2(g_0^{-1})$ are proximal simultaneously.
This imply in turn that $\rho_1(g_0)$, $\rho_1(g_0^{-1})$,
$\rho_2(g_0)$ and $\rho_2(g_0^{-1})$ are also proximal.

We can set $A_1=\{g_0\}$ and get the claim for $m=1$.
We proceed by induction,
assume that we can construct $A_m$ for some $m\ge1$.
We try to find an element $h\in\Ga$ such that
$A_{m+1}:=A_m\cup\{hg_0h^{-1}\}$ is generic.
Clearly $\bar z_{\rho_1(hg_0h^{-1})}=\rho_1(h)\bar z_{\rho_1(g_0)}$.
One condition $h$ needs to satisfy is that neither
$\rho_1(h) z_{\rho_1(g_0)}$ nor $\rho_1(h) z_{\rho_1(g_0^{-1})}$
should belong to those proper subspaces $V$ of $\fsl_d(\C)$ for which
\[
|\{g\in A_m\cup\wt A_m\sep z_{\rho_1(g)}\in V\}|\ge \dim V.
\]
There are a finite number of such subspaces, hence this is a Zariski
open condition on $\s_1(h)$.
It can be seen in a similar fashion that
$A_{m+1}$ is generic if $(\sigma_1(h),\sigma_2(h))$ belongs to a certain
Zariski dense open subset of $SL_d(\C)\times SL_d(\C)$, and
the lemma follows by induction.
\end{proof}
We remark that it is easy to see from the proof that $A_m$
can be chosen in such a way that it is generic with respect to any
pair of embeddings $\s_1$ and $\s_2$.

\begin{lem}
\label{lem_qp}
Let $A\subset\Ga$ be a generic set of cardinality at least $(d^2+2)/2$.
Then for each $g\in A\cup\wt A$ and $i\in\{1,2\}$, there is a
neighborhood $U_g^{(i)}\subset\P(\fsl_d(\C))$ of $\bar z_{\rho_i(g)}$ 
with the following property.
For any invertible linear transformation $T$ on
$\fsl_d(\C)$ there is a $g\in A\cup\wt A$ such that
$T(U_g^{(1)})\cap U_g^{(2)}=\emptyset$.
\end{lem}
First we recall \cite[Proposition 2.1]{CaSe}.
Let $T_1,T_2,\ldots$ be a sequence of invertible linear transformations
on $\fsl_d(\C)$.
There is a not necessarily invertible linear transformation $T\neq0$
and a subsequence of $T_1,T_2,\ldots$ that considered as maps
on $\P(\fsl_d(\C))$
converge uniformly to $T$ on compact subsets of
$\P(\fsl_d(\C))\sm \overline{{\rm Ker}(T)}$. 
\begin{proof}[Proof of Lemma \ref{lem_qp}]
Assume to the contrary that the claim is false.
Then there is a sequence $\{T_k\}$ of linear transformations such that for
any choice of the neighborhoods $U_g^{(i)}$ ($i\in\{1,2\}$ and
$g\in A\cup\wt A$),
we have $T_k(U_g^{(1)})\cap U_g^{(2)}\neq\emptyset$ for $k$ large enough.
By the aforementioned result, we may assume that $\{T_k\}$ converges
uniformly to a linear transformation $T$ on compact subsets of
$\P(\fsl_d(\C))\sm \overline{{\rm Ker}(T)}$.
This implies that if $z_{\rho_1(g)}\notin{\rm Ker}(T)$, then
$T(\bar z_{\rho_1(g)})=\bar z_{\rho_2(g)}$.
When $T$ is invertible, this violates (iii) in the definition
of generic sets.
If $T$ is not invertible, we get a contradiction with (ii) of that
definition, either for $V={\rm Ker} (T)$ or for $V={\rm Im } (T)$,
and the lemma follows.
\end{proof}

\begin{lem}
\label{lem_free}
Let $A\subset\Ga$ be generic, and for each
$g\in A\cup\wt A$ and $i\in\{1,2\}$
let $U_g^{(i)}\subset\P(\fsl_d(\C))$
be a sufficiently small neighborhood of $\bar z_{\rho_i(g)}$.
Then there is a positive integer $M$ such that $\{g^M\sep g\in A\}$
freely generates a subgroup of $\Ga$ and if $h=g_1^Mg_2^M\cdots g_k^M$
is a reduced word, then $\rho_1(h)$ and $\rho_2(h)$ are proximal
with $\bar z_{\rho_i(h)}\in U_{g_1}^{(i)}$.
\end{lem}
\begin{proof}
To simplify the notation we omit those subscripts and superscripts
that indicate which of the representations $\rho_1$ or $\rho_2$
the object in question is related to.
If $U_g$ are sufficiently small, then there are compact sets
$Q_g\subset\P(\fsl_d(\C))\sm\overline {V_{\rho(g)}}$ for
$g\in A\cup\wt A$ and an integer $M$ such that the following hold:
\bean
d(\rho(g^M)\bar x,\rho(g^M)\bar y)&<&d(\bar x,\bar y)
\quad{\rm for}\quad x,y\in Q_g\quad{\rm and}\\
U_{g'}&\subset& Q_g\quad{\rm if}\quad gg'\neq 1.
\eean
Here we used property (i) of generic sets.
If $M$ is large enough we clearly have $\rho(g^M)Q_g\subset U_g$
also.
By induction, we see that if $h=g_1^M\cdots g_k^M$ is a reduced word then
$\rho(h)Q_{g_k}\subset U_{g_1}$, and 
$d(\rho(h)\bar x,\rho(h)\bar y)<d(\bar x,\bar y)$ for
$\bar x,\bar y\in Q_{g_k}$.
If $g_1g_k\neq 1$, then $U_{g_1}\subset Q_{g_k}$ and the claim
follows for $h$ by the aforementioned lemma of Tits
\cite[Lemma 3.8 (ii)]{Tit}.
If $g_1g_k=1$, then write $h=g_1^Mh'g_1^{-M}$.
If $h'$ is proximal with $\bar z_{\rho(h')}\in U_{g_2}$, then
$h$ is also proximal with $\bar z_{\rho(h)}=\rho(g_1)\bar z_{\rho(h')}$,
and the claim follows by induction.
Now $\{g^M\sep g\in A\}$ generates freely a
group since the identity is not proximal.
\end{proof}

\begin{proof}[Proof of Proposition \ref{prp_trans}]
Let notation be as in the statement of the proposition.
Let $A$ be a generic set of cardinality $m\ge(d^2+2)/2$, and
set $S'=\{g^M\sep g\in A\cup\wt A\}$, where $M$ is the
same as in Lemma \ref{lem_free}.
For $g\in S'$ and $i\in\{1,2\}$ let $U_g^{(i)}$ be a neighborhood
of $\bar z_{\rho_i(g)}$ which is sufficiently small for Lemmata
\ref{lem_qp} and \ref{lem_free}.
Then there is an element $g_0\in S'$ such that
$T(U_{g_0}^{(1)})\cap U_{g_0}^{(2)}=\emptyset$.
For $h\in H_T$ we clearly have $T\bar z_{\rho_1(h)}=\bar z_{\rho_2(h)}$,
so if $h$ is a reduced word of form $g_1\cdots g_k$ with $g_i\in S'$,
then $g_1\neq g_0$ by Lemma \ref{lem_free}.
If $h\in SL_d(\OO_K)$, a similar result holds for
$hH_Th^{-1}=H_{\rho_2(h)T\rho_1(h^{-1})}$.
Therefore by Lemma \ref{lem_alg}, we have
\[
|B_l\cap H_T|\le(2m-1)^{l/2+1}(2m-2)^{l/2-1},
\]
where $B_l$ is the set of reduced words of length $l$ over the
alphabet $S'$.

Set $P_k(l)=\chi_{S'}^{(2k)}(w)$, where $w\in B_l$.
Since $|B_l|=2m(2m-1)^{l-1}$ for $l\ge1$,
\be
\label{eq_psi}
1=P_k(0)+\sum_{l\ge1}2m(2m-1)^{l-1}P_k(l).
\ee
By a result of Kesten \cite[Theorem 3.]{Ke}, we have
\[
\limsup_{k\to\infty} (P_k(0))^{1/k}=(2m-1)/m^2.
\]
From general properties of Markov chains (see \cite[Lemma 1.9]{Wo})
it follows that
\[
P_k(0)\le\left(\frac{2m-1}{m^2}\right)^k.
\]
Since $\chi_{S'}^{(2k)}$ is
symmetric, we have $P_k(0)=\sum_g\chi_{S'}^{(k)}(g)^2$, hence
$P_k(l)\le P_k(0)$ for all $l$
by the Cauchy-Schwartz inequality.
Now we can write
\bean
\chi_{S'}^{(2k)}(H_T)&=&\sum_l|B_l\cap H_T|P_k(l)\\
&\le&\sum_l(2m-1)^{l/2+1}(2m-2)^{l/2-1}P_k(l)\\
&\le&\sum_{l\le k/10}(2m-1)^{l/2+1}(2m-2)^{l/2-1}
\left(\frac{2m-1}{m^2}\right)^k\\
&&+\left(\frac{2m-1}{2m}\right)^{k/20}
\sum_{l\ge k/10}2m(2m-1)^{l-1}P_k(l)\\
&<&\left(\frac{2m-1}{2m}\right)^{k/2}+\left(\frac{2m-1}{2m}\right)^{k/20},
\eean
which was to be proven.
The inequality between the third and fourth lines follows form
(\ref{eq_psi}).
\end{proof}

\begin{proof}[Proof of Theorem \ref{thm_escp}]
Let $S'$ be the same as in Proposition \ref{prp_trans} and let $C$ and $\d$
be the same as in Corollary \ref{cor_alg}.
As we remarked after Lemma \ref{lem_generic}, we can choose
$S'$ in such a way that it works for any pair of embeddings $\s_1$
and $\s_2$.
There is a constant $c$ depending on the set $S'$ such that
$\log\|\wh\s (g)\|\le cl$ for $g\in\prod_l S'$.
Then for $l=\d\log [SL_d(\OO_k/(q)):H^\sharp]/c$,
we have
\[
\pi_q[\chi_{S'}^{(l)}](H^\sharp)=\chi_{S'}^{(l)}(\LL_\d(H^\sharp)).
\]
Combining Corollary \ref{cor_alg} with either Proposition \ref{prp_subspc}
or Proposition \ref{prp_trans} we get
\[
\chi_{S'}^{(l)}(\LL_\d(H^\sharp))\ll[SL_d(\OO_k/(q)):H^\sharp]^{-\d c'}
\]
with some $c'>0$.
If $l$ is even, then by the symmetry of $S'$,
\[
(\pi_q[\chi_{S'}^{(l/2)}](gH^{\sharp}))^2\le\pi_q[\chi_{S'}^{(l)}](H^\sharp)
\]
for any coset $gH^\sharp$, and
by $[H:H^\sharp]<C^n$ we then have
\[
\pi_q[\chi_{S'}^{(l/2)}](H)\le C^n(\pi_q[\chi_{S'}^{(l)}](H^\sharp))^{1/2}.
\]
If $l_1\le l_2$, then clearly
\[
\pi_q[\chi_{S'}^{(l_2)}](H)\le
\max_{g}\pi_q[\chi_{S'}^{(l_1)}](gH).
\]
Now it is straightforward to get the theorem by putting together the
above inequalities.
\end{proof}

\section{A product theorem}
\label{sec_prod}

Recall that $H_1\la_L H_2$ is a shorthand for
$[H_1:H_1\cap H_2]\le L$.
We denote by $Z(G)$ the center of the group $G$,
by $\CC(g)$ the centralizer
of the element $g\in G$ and by $\NN_G(H)$ the normalizer
of the subgroup $H<G$.
In this section $K$ is not a number-field, it usually stands for
a large positive real number.
We begin by listing the assumptions already mentioned
in Theorem \ref{thm_prod}.
When we say that something depends on the constants appearing
in the assumptions (A1)--(A5) we mean $L$ and the function
$\d(\e)$ for which (A4) holds.

\begin{itemize}
\item[(A0)]
$G=G_1\times\cdots\times G_n$ is a direct product, and
the collection of the factors satisfy (A1)--(A5) for some
sufficiently large constant $L$.

\item[(A1)]
There are at most $L$ isomorphic copies of the same group in the
collection.

\item[(A2)]
Each $G_i$ is quasi-simple and we have
$|Z(G_i)|<L$.

\item[(A3)]
Any nontrivial representation of $G_i$ is of dimension at least
$|G_i|^{1/L}$.

\item[(A4)]
For any $\e>0$, there is a $\d>0$ such that the following holds.
If $\mu$ and $\nu$ are probability measures on $G_i$ satisfying
\[
\|\mu\|_2>|G_i|^{-1/2+\e}
\quad{\rm and}\quad
\mu(gH)<|G_i|^{-\e}
\] 
for any $g\in G_i$ and for any proper $H<G_i$, then
\be
\label{eq_assump2}
\|\mu*\nu\|_2\ll\|\mu\|_2^{1/2+\d}\|\nu\|_2^{1/2}.
\ee

\item[(A5)]
For some $m<L$, there are classes $\HH_0,\HH_1,\ldots,\HH_m$ of
subgroups of $G_i$ having the following properties.
\begin{itemize}

\item[$(i)$]
$\HH_0=\{Z(G)\}$.

\item[$(ii)$]
Each $\HH_j$ is closed under conjugation by elements of $G_i$.

\item[$(iii)$]
For each proper $H<G_i$ there is an $H^\sharp\in \HH_j$ for some $j$
with $H\la_L H^\sharp$.

\item[$(iv)$]
For every pair of subgroups $H_1,H_2\in\HH_j$, $H_1\neq H_2$
there is some $j'<j$ and $H^\sharp\in\HH_{j'}$
for which $H_1\cap H_2\la_L H^{\sharp}$.
\end{itemize}
\end{itemize}

We remark that considering the induced representation,
(A3) implies that for any proper subgroup $H<G_i$
we have
\be
\label{eq_index}
[G_i:H]>|G_i|^{1/L}.
\ee
One may think about (A5) that there is a notion for dimension of
the subgroups of $G_i$.

In the next section we show that Theorem \ref{thm_prod}
is a simple corollary of the following seemingly weaker
result.
\begin{prp}
\label{prp_prod}
Let $G$ be a group satisfying (A0)--(A5).
For any $\e>0$ there is a $\d>0$ depending only on $\e$
and the constants in assumptions such that the following holds.
If $S\subset G$ is symmetric such that
\[
|S|<|G|^{1-\e}\quad {\rm and}\quad \chi_S(gH)<[G:H]^{-\e}|G|^{\d}
\]
for any $g\in G$ and any proper $H<G$,
then $|\prod_3 S|\gg|S|^{1+\d}$.
\end{prp}

\subsection{Proof of Theorem \ref{thm_prod} using Proposition
\ref{prp_prod}}
\label{sec_th3}

We make use of the following result which appeared
first implicitly in the proof of Proposition 2 in Bourgain,
Gamburd \cite{BG1}.
\begin{lem}[Bourgain, Gamburd]
\label{lem_bszg}
Let $\mu$ and $\nu$ be two probability measures on an arbitrary
group $G$ and let $K>2$ be a number.
If
\[
\|\mu*\nu\|_2>\frac{\|\mu\|_2^{1/2}\|\nu\|_2^{1/2}}{K}
\]
then there is a symmetric set $S\subset G$ with
\[
\frac{1}{K^R\|\mu\|_2^2}\ll |S|\ll \frac{K^R}{\|\mu\|_2^2},
\]
\[
|{\textstyle\prod_3 S}|\ll K^R|S|\quad {\rm and}
\]
\[
\min_{g\in S}\left(\wt \mu *\mu\right)(g)\gg \frac{1}{K^R|S|},
\]
where $R$ and the implied constants are absolute.
\end{lem}

\begin{proof}
We include the proof only for the sake of completeness,
the argument is essentially the same as in the proof of 
\cite[Proposition 2]{BG1}.

First we note that by Young's inequality
$\|\mu*\nu\|_2\le\|\mu\|_2$ and hence
$\|\nu\|_2<K^2\|\mu\|_2$ and similarly $\|\mu\|_2<K^2\|\nu\|_2$.
Let $\l$ be a nonnegative measure with $\|\l\|\le1$ and
$\|\l\|_2^2<c$.
Observe that if $\l(g)\ge K'c$ for some $K'$ for every $g\in\supp\l$,
then $\|\l\|_2^2\ge K'c\|\l\|_1$, hence $\|\l\|_1<1/K'$.
Similarly, if $\l(g)\le c/K'$ for all $g$, then $\|\l\|_2^2<c/K'$.
Now define the sets
\bean
A_i&=&\{g\in G\sep 2^{i-1}\|\mu\|_2^2<\mu(g)
\le 2^{i}\|\mu\|_2^2\}\quad{\rm and}\\
B_i&=&\{g\in G\sep 2^{i-1}\|\nu\|_2^2<\nu(g)\le 2^{i}\|\nu\|_2^2\}
\eean
for $|i|<10 \log K$.
By Young's inequality,
\[
\|\mu*\nu\|_2\le\sum_{|i|,|j|\le10\log K}
2^{i+j}\|\mu\|_2^2\|\nu\|_2^2|A_i||B_j|
\|\chi_{A_i}*\chi_{B_j}\|_2
+K^{-5}(\|\mu\|_2+\|\nu\|_2),
\]
hence there must be a pair of indices $i,j$ such that
\be\label{eq_mesconv1}
2^{i+j}\|\mu\|_2^2\|\nu\|_2^2\||A_i||B_j|\|\chi_{A_i}*\chi_{B_j}\|_2
\gg\frac{\|\mu\|_2^{1/2}\|\nu\|_2^{1/2}}
{K\log^2 K}.
\ee
By construction, for $g\in A_i$ we have
\[
2^i\|\mu\|_2^2\gg\mu(g)\gg2^i\|\mu\|_2^2,
\]
and by (\ref{eq_mesconv1}) and Young's inequality,
$1\ge\mu(A_i)\gg1/K^R$.
Here, and everywhere $R$ denotes an absolute constant
which need not be the same at different occurrences.
These together give
\[
\frac{K^R}{\|\mu\|_2^2}\gg|A_i|\gg\frac{1}{K^R\|\mu\|_2^2}.
\]
We may get the analogous inequalities
\[
\frac{K^R}{\|\mu\|_2^2}\gg|B_j|\gg\frac{1}{K^R\|\mu\|_2^2}.
\]
in a similar way and using the relations between $\|\mu\|_2$
and $\|\nu\|_2$.
Applying our inequalities to $(\ref{eq_mesconv1})$,
we get
\[
\|\chi_{A_i}*\chi_{B_j}\|_2^2\gg\frac{1}{K^R|A_i|^{1/2}|B_j|^{1/2}}.
\]
We invoke the non-commutative version of the Balog-Szemer\'edi-Gowers
theorem proven by Tao \cite[Theorem 5.2]{Tao}, (note that we use
a different normalization).
This gives subsets $A\subset A_i$ and $B\subset B_i$
with $|A|\gg|A_i|/K^R$ and $|A.B|\ll K^R|A|^{1/2}|B|^{1/2}$.
Ruzsa's triangle inequality \cite[Lemma 3.2]{Tao}
for the sets $A$ and $\wt B$
gives $|A.\wt A|\ll K^R|A|$.
Using \cite[Proposition 4.5]{Tao} with $n=3$, we get a
symmetric set $S$ with $|S|>|A|/K^R$ and
\[
|\textstyle\prod_3 S|\ll K^R|A|\ll K^{R'}|S|.
\]
In the proof of Proposition 4.5 of \cite{Tao}
the set $S$ is defined by
\[
\{g\in G\sep|A\cap(A.\{g\})|>|A|/C\}
\]
with $C=2|A.\wt A|/|A|$.
For $g\in S$, we have
\[
\left(\wt \mu*\mu\right)(g)\ge2^{2i-2}\|\mu\|_2^4|A\cap(A.\{g\})|
\gg\frac{1}{K^R|S|}.
\]
The expression in the middle is bounded below by $\|\mu\|_2^2/K^R$ also,
which gives the required upper bound for $|S|$, since
$\|\wt \mu*\mu\|_1=1$.
\end{proof}

\begin{proof}[Proof of Theorem \ref{thm_prod}]
Assume that the conclusion of the theorem fails, i.e.
that there is an $\e$ such that for any $\d$ there
are probability measures $\mu$ and $\nu$ with
\[
\|\mu\|_2>|G|^{-1/2+\e}\quad{\rm and}\quad
\mu(gH)<[G:H]^{-\e}
\]
for any $g\in G$ and for any proper $H<G$, and yet
\[
\|\mu*\nu\|_2\ge\|\mu\|_2^{1/2+\d}\|\nu\|_2^{1/2}.
\]
Take $K=\|\mu\|_2^{-\d}$ in Lemma \ref{lem_bszg}.
Note that by the third property of the set, $S$
we have
\[
\chi_S(gH)\ll K^R\wt \mu*\mu(gH)\le K^R\max_{h\in G} \mu(hH)
\ll|G|^{R\d}[G:H]^{-\e}.
\]
Now $|\prod_3 S|\ll K^R|S|$ contradicts Proposition \ref{prp_prod},
if $\d$ is small enough.
\end{proof}

\subsection{Proof of Proposition \ref{prp_prod}} 
\label{sec_prp14}

Throughout sections \ref{sec_prp14}--\ref{sec_ld}, we assume
that $G=G_1\times\ldots\times G_n$ satisfies (A0)--(A5) with
some $L$.
$\e$ and $S$ are the same as in Proposition \ref{prp_prod}, and
we fix a sufficiently small $\d$.
By sufficiently small, we mean that we are free to use inequalities
$\d<\d'$, where $\d'$ is any function of $\e$ and the constants in
(A1)--(A5).
We use $c,\d',\d'',Q,Q'$, etc. to denote positive constants that may
depend only on $\e$ and the constants in (A1)--(A5).
These need not be the same at different occurrences.
We will also use inequalities of the form
\be
\label{eq_size}
Q\log|G_i|<|G_i|^{\d\d'}.
\ee
Let $N$ be the product of those factors $G_i$, for which
such an inequality fails.
Since the same group appears at most $L$ times among the $G_i$,
the size of $N$ is bounded.
Replace $G$ by $G/N$.
For any $H<G/N$, we have $[G/N:H]=[G:HN]$ and if $\bar S$ denotes
the projection of $S$ in $G/N$, then we have
$|\prod_3 S|\ge|\prod_3 \bar S|$ and $|S|\le|\bar S||N|$.
Hence the theorem for the group $G/N$ implies itself for $G$
with a larger implied constant.
Thus we can use (\ref{eq_size}) without loss of generality.

In a similar fashion we may replace each $G_i$ by $G_i/Z(G_i)$,
hence from now on, we assume that all the $G_i$ are simple.
This may introduce a factor of size at most $L^n$ which is $\ll |G|^\d$
for any $\d>0$.

We follow the argument of Bourgain, Gamburd and Sarnak
\cite[Section 5]{BGS}.
First we introduce some notation.
Denote $\pi_i$ for $1\le i\le n$ the projection
from $G$ to $G_i$.
Set $G_{\le i}=\times_{j\le i}G_i$ and denote
$\pi_{\le i}$ the projection from $G$ to $G_{\le i}$.
To the set $S$, we associate a tree of $n+1$ levels.
Level 0
consists of a single vertex, while for $i>0$ the vertices of level
$i$ are the elements of the set $\pi_{\le i}(S)$,
and a vertex $g$ on level $i-1$ is connected to those
vertices on level $i$ which are of the form $(g, h)$ with
some $h\in G_{i}$.
By removing some vertices, we can get a regular tree, that is a
tree which has vertices of equal degree on each level.
More precisely, using
\cite[Lemma 5.2]{BGS} we obtain a subset $A\subset S$
and a sequence $\{D_i\}_{1\le i\le n}$ of positive integers
with $D_i\ge|G_i|^\d$ or $D_i=1$ such that for any
$g\in\pi_{\le i-1}A$, we have
\[
|\{h\in G_{i}\sep (g,h)\in\pi_{\le i}(A)\}|=D_{i},
\]
and
\be
\label{eq_A}
|A|>\left[\prod_{i=1}^n(|G_i|^\d\log|G_i|)\right]^{-1}|S|>|G|^{-2\d}|S|.
\ee
The second inequality in (\ref{eq_A}) is of type (\ref{eq_size}).

We briefly outline the proof.
Consider the set $\prod_k A$ for some integer $k$ and the tree
associated to it in the way described above.
If $g\in\pi_{\le i-1}\left(\prod_k A\right)$ is a vertex on level $i-1$ and
$g=g_1\cdots g_k$ with $g_l\in\pi_{\le i-1}(A)$,
then $(g,h)$ is connected to $g$ for every $h$ in the product-set
\[
\{h_1\sep(g_1,h_1)\in\pi_{\le i}(A)\}\ldots
\{h_k\sep(g_k,h_k)\in\pi_{\le i}(A)\}.
\]
Let $I_s$ be the set of indices $1\le i\le n$ for which
$D_{i}<|G_{i}|^{1-1/3L}$ (i.e. indices corresponding to small
degrees), for such an index, there is hope that
we can apply (A4) for $G_i$ and get that the above product-set
is of size $D_{i}^{1+\d'}$ for some $\d'>0$.
We make this speculation precise
in section \ref{sec_sd}.
Set $I_l=\{1,\ldots,n\}\sm I_s$ (indices corresponding to large degrees),
$G_s=\times_{i\in I_s}G_i$ and
$G_l=\times_{i\in I_l}G_i$, and denote by
$\pi_s$ and $\pi_l$ the projections from $G=G_s\times G_l$
to $G_s$ and $G_l$ respectively.
We get from a result of Gowers \cite{Gow} that
$\pi_l(S.S.S)=G_l$.
In subsection \ref{sec_ld}, we prove using a result
of Farah \cite{Far} on approximate homomorphisms that
$\pi_l^{-1}(1)\cap\prod_9 S$ contains an element $g$
whose centralizer $\CC(g)$ is of large 
index.
Then $S$ will contain elements from at least
$[G:\CC(g)]^{\e}|G|^{-\d}$ cosets of $\CC(g)$, hence
there are many $h\in\prod_{11} S$ with $\pi_l(h)=1$,
and $\prod_{12} S$ is much larger than $G_l$.

Finally, we mention that there is a useful result of
Helfgott \cite[Lemma 2.2]{Hel}
that allows us to bound $|S.S.S|$ in terms of larger iterated
product-sets.
He proves that if
$S$ is a symmetric subset of an arbitrary group $G$ and $k\ge 3$
is an integer, then
\be
\label{eq_iterated}
\frac{|{\textstyle\prod_k}S|}{|S|}\le
\left(\frac{|S.S.S|}{|S|}\right)^{k-2}.
\ee

\subsection{The case of many small degrees}
\label{sec_sd}

In this section we prove
\begin{prp}
\label{prp_sd}
There are positive constants $\d'$ and
$Q$ depending only on $\e$ and
the constants in the assumptions, such that 
\[
{|\textstyle\prod_{2^{m+1}}S|}>|S||G|^{-Q\d}
\prod_{i\in I_s}D_i^{\d'},
\]
where $m$ is the same as in (A5).
\end{prp}

The biggest issue here is that beside its size, we have no
information about a set of form $\{b\sep(a,b)\in\pi_{\le i}(A)\}$.
A large part of it might be contained in a coset of a proper
subgroup and then (A4) does not apply with $\mu$ being the
normalized counting measure on that set.
To resolve this problem, we multiply sets of this form together
with random elements of $G_i$.
We need to construct a probability distribution supported on $S$
whose projection to most factors $G_i$ is well-behaved in the following sense.
\begin{lem}
\label{lem_B}
There is a subset $B\subset S$, and there is a partition of the indices
$1,\ldots,n$ into two parts $J_g$ and $J_b$ such that
\be
\label{eq_B2}
\prod_{i\in J_b}|G_i|\le|G|^{\d/\d'},
\ee
and for any $i\in J_g$ and for any proper coset $gH\subset G_i$,
we have
\be
\label{eq_B1}
\chi_B(\{x\in G\sep\pi_i(x)\in gH\})\le|G_i|^{-\d'},
\ee
where $\d'>0$ is a constant
depending on $\e$ and on $L$.
\end{lem}
\begin{proof}
We obtain the set $B$ by the following algorithm.
First set $B=S$ and $J_g=\{1,\ldots,n\}$.
Then iterate the following step as long as possible.
If there is an index $i\in J_g$ and a coset $gH\subset G_i$
such that (\ref{eq_B1}) fails, then replace $B$ by
\[
\{x\in B\sep\pi_i(x)\in gH\}
\]
and put $i$ into $J_b$.
It is clear that (\ref{eq_B1}) holds when this process terminates.
As for (\ref{eq_B2}), note that
\[
\chi_S(B)\ge\prod_{i\in J_b} |G_i|^{-\d'}
\]
and $B$ is contained in a coset of a subgroup of index at least
$\prod_{i\in J_b}|G_i|^{1/L}$ by (\ref{eq_index}).
These together and the assumption of Proposition \ref{prp_prod} on $S$
imply
\[
\prod_{i\in J_b} |G_i|^{-\d'}<
\left(\prod_{i\in J_b}|G_i|^{1/L}\right)^{-\e}|G|^{\d},
\]
and (\ref{eq_B2}) follows easily if we set $\d'=\e/2L$.
\end{proof}

Now assume that $i\in J_g$.
Then, starting from arbitrary sets $A_1,\ldots,A_{2^m}\subset G_i$
of the same size $|G_i|^\d<D<|G_i|^{1-1/3L}$, we construct a measure $\l_m$
for which (A4) is applicable.

Choose the elements $x_j$ for $1\le j\le 2^m-1$ independently
at random according to the distribution $\chi_B$.
Set $y_j=\pi_i(x_j)$.
For $0\le k\le m$ define
\[
\l_k=\chi_{A_1}*1_{y_1}*\chi_{A_2}*1_{y_2}*\ldots
*1_{y_{2^k-1}}*\chi_{A_{2^k}},
\]
where $1_y$ denotes
the unit mass measure at $y$.

\begin{lem}
\label{lem_lambda}
If $i\in J_g$, then there is a constant $\d'$
depending only on $\e$ and $L$ such that
the probability of the event that
\be
\label{eq_lambda}
\l_k(gH)<D^{-\d'/10^k}
\ee 
holds for any proper coset $gH<G_i$, if $H\in\HH_l$ for some
$l\le k$ is at least
\be
\label{eq_prob}
1-(2^k-1)|G_i|^{-\d'}.
\ee

\end{lem}
\begin{proof}
Let $\d'$ be twice the $\d'$ of the previous lemma.
For $k=0$, the claim follows from $L/D<D^{-\d'}$ which is
an inequality of form (\ref{eq_size}).
We assume that $k>0$
and that the claim holds for $k-1$.
Set
\[
\eta_{k-1}=\chi_{A_{2^{k-1}+1}}*1_{y_{2^{k-1}+1}}*
\chi_{A_{2^{k-1}+2}}*1_{y_{2^{k-1}+2}}*\ldots
*1_{y_{2^k-1}}*\chi_{A_{2^k}}
\]
and assume that $y_1,\ldots,y_{2^{k-1}-1}$
and $y_{2^{k-1}+1},\ldots,y_{2^k-1}$ are chosen in such a way
that $\l_{k-1}$ and $\eta_{k-1}$ satisfies
\[
\l_{k-1}(gH)<D^{-\d'/10^{k-1}}\quad{\rm and}\quad
\eta_{k-1}(gH)<D^{-\d'/10^{k-1}}
\]
for subgroups $H\in\HH_{k-1}$.
By the induction hypothesis, the probability of such a choice is at
least $1-(2^k-2)|G_i|^{-\d'}$.
Now assume that $\l_k=\l_{k-1}*1_{y_{2^{k-1}}}*\eta_{k-1}$
violates (\ref{eq_lambda}) for some $g\in G_i$ and $H\in\HH_k$.
To shorten the notation write $y=y_{2^{k-1}}$.
We prove that $y$ is in a set of $\pi_i(\chi_B)$ measure at most
$|G_i|^{-\d'}$,
and this set will depend only on $\l_{k-1}$ and $\eta_{k-1}$,
in particular it will be
independent of the choice of $H$ and $g$. 
Let $\{h_j\}$ be a left transversal for $H$ (i.e. a system of representatives
for left $H$--cosets).
Then it is easy to see that $\{gh^{-1}_j\}$ is a right transversal
for $gHg^{-1}$, hence
\[
\l_k(gH)=\sum_j\l_{k-1}(gHg^{-1}gh_j^{-1})\eta_{k-1}(y^{-1}h_jH)
\]
We claim that for some index $j$, we have
\be
\label{eq_bc}
\l_{k-1}(B_j)\ge D^{-\d'/10^k}/2\quad{\rm and}\quad
\eta_{k-1}(C_j)\ge D^{-\d'/10^k}/2,
\ee
where $B_j=gHh_j^{-1}$ and $C_j=y^{-1}h_j H$.
Assume to the contrary that this fails.
Then we have
\bean
\sum_j\l_{k-1}(B_j)\eta_{k-1}(C_j)
&=&\sum_{j:\l_{k-1}(B_j)< D^{-\d'/10^k}/2}\l_{k-1}(B_j)\eta_{k-1}(C_j)\\
&&+\sum_{j:\eta_{k-1}(C_j)<D^{-\d'/10^k}/2}\l_{k-1}(B_j)\eta_{k-1}(C_j)\\
&<& D^{-\d'/10^k},
\eean
a contradiction.

Let $j$ be such that (\ref{eq_bc}) holds.
Define $H_1=h_jHh_j^{-1}$ and
$H_2=y^{-1}H_1y$.
Notice that $\wt B_j.B_j\subset H_1$ and
$C_j.\wt C_j\subset H_2$.
This shows that there are subgroups $H_1,H_2\in\HH_k$
such that
\be
\label{eq_h1h2}
(\wt\l_{k-1}*\l_{k-1})(H_1)\ge D^{-2\d'/10^k}/4\quad{\rm and}
\quad
(\eta_{k-1}*\wt\eta_{k-1})(H_2)\ge D^{-2\d'/10^k}/4
\ee
and
$H_1=yH_2y^{-1}$.
For fixed $H_1$ and $H_2$, this restricts $y$ to a single
$\NN(H_2)$--coset.
By Lemma \ref{lem_B}, this is a set of $\chi_B$ measure at most
$|G_i|^{\d'/2}$.
The final step is to show that the number of possible pairs
$H_1,H_2$ such that (\ref{eq_h1h2}) holds is at most $|G_i|^{\d'/2}$.

Suppose that we have $M$ distinct subgroups $H_1\in\HH_k$
such that 
\[
\wt \l_{k-1}*\l_{k-1}(H_1)\ge D^{-2\d'/10^k}/4.
\]
If $H_1$ and $H_1'$ are two such subgroups, then
$H_1\cap H_1'\la_L H^\sharp$ for some $H^\sharp\in\HH_{k-1}$.
By the induction hypothesis, we have
$\wt \l_{k-1}*\l_{k-1}(H^\sharp)\le D^{-\d'/10^{k-1}}$,
hence $\wt\l_{k-1}*\l_{k-1}(H_1\cap H_2)\le LD^{-\d'/10^{k-1}}$.
By the inclusion-exclusion principle, we have
\[
MD^{-2\d'/10^k}/4-M^2LD^{-\d'/10^{k-1}}\le1.
\]
This is violated if $M=D^{\d'/4\cdot 10^{k-1}}$,
in fact we need $D^{\d'/2\cdot 10^k}>4(1+L)$,
which is an inequality of form (\ref{eq_size}).
Thus $M<D^{\d'/4\cdot 10^{k-1}}$, and
as the case of $H_2$ is similar, the proof is complete.
\end{proof}

Using property (A4), we get the following
simple
\begin{cor}
\label{cor_l}
Assume that $|G_i|^\d<D<|G_i|^{1-1/3L}$,
and let $A'\subset G_i$ be any set of cardinality $D$.
There is a positive number $\d'$ depending only on $\e$
and the constants in (A1)--(A5) such that for the above
defined $\l_m$, we have
\[
\|\l_m*\chi_{A'}\|_2\ll D^{-1/2-\d'}
\]
with probability at least $1/2$.
\end{cor}
\begin{proof}
By Lemma \ref{lem_lambda} (and using (\ref{eq_size})),
we have with probability at least $1/2$
that $\l_m(gH)<LD^{-\d''}$ with some $\d''>0$ for every proper
coset $gH$.
By (\ref{eq_size}), we have $L<D^{-\d''/2}$.
If say $\|\l_m\|_2>|G_i|^{-1/2+1/12 L}$, then we get 
\[
\|\l_m*\chi_{A'}\|_2\le\|\l_m\|_2^{1/2+\d'}\|\chi_{A}'\|_2^{1/2}
\]
by (A4) with $\mu=\l_m$ and $\nu=\chi_{A'}$.
Otherwise the claim is trivial by Young's inequality.
\end{proof}

In what follows, we need some basic facts about entropy.
Let $\mu$ be a probability measure on $G$, and let $\AA$
be a partition of $G$.
The entropy of $\AA$ is defined by
\[
H_\mu(\AA)=\sum_{A\in \AA}-\mu(A)\log(\mu(A)),
\]
with the convention $0\cdot \log 0=0$.
We also use the notation $H_\mu$ for the entropy of the partition
consisting of one element sets.
The inequalities
\[
|\supp \mu|\ge e^{H_\mu}\ge\frac{1}{\|\mu\|_2^2}
\]
are well-known.
If $B\subset G$, we write $\mu\restr{B}(A)=\mu(A\cap B)/\mu(B)$, and
if $\BB$ is another partition, we define the conditional entropy
by
\[
H_\mu(\AA| \BB)=\sum_{B\in \BB}H_{\mu\restr B}(\AA)\mu(B).
\]
It is easy to see that
\[
H_\mu(\AA\lor\BB)=H_\mu(\AA| \BB)+H_\mu(\BB),
\]
where $\AA\lor \BB$ denotes the coarsest partition that is finer
than both $\AA$ and $\BB$.
On finite sets, partitions and $\sigma$-algebras are essentially the
same, hence we make no distinction.

Finally, we turn to the
\begin{proof}[Proof of Proposition \ref{prp_sd}]
First we introduce a couple of $\sigma$-algebras (partitions)
on the set $A^{\times (2^m+1)}$, i.e. on the $2^m+1$-fold Cartesian
product of $A$.
Let $\AA_i$ be the coarsest $\sigma$-algebra, for which
the projection map
\[
\pi_{\le i}:A^{\times (2^m+1)}\to G_{\le i}^{\times (2^m+1)}
\]
is measurable.
Furthermore, let $\BB$  be the coarsest $\sigma$-algebra, for which
the map
\[
(a_1,\ldots,a_{2^m},a_{2^m+1})\mapsto a_1x_1a_2x_2\cdots x_{2^m -1}
a_{2^m}a_{2^{m}+1}
\]
is measurable, where the elements $x_1,\ldots,x_{2^m-1}$ are chosen
independently at random according to the distribution $\chi_B$,
hence the partition $\BB$ is random.
Denote by $\mu$ the measure $\chi_A^{\otimes (2^m+1)}$
on $A^{\times(2^m+1)}$.
It follows from the definition that the entropy of the measure
\[
\chi_A*1_{x_1}*\chi_A*1_{x_2}*\ldots *1_{x_{2^m-1}}*\chi_A*\chi_A
\]
equals $H_\mu(\BB)$.
We write for the expectation of $H_\mu(\BB)$:
\bean
\E[H_\mu(\BB)]
&\ge&\sum_{i=1}^{n}\E[H_\mu(\BB\land\AA_i|\AA_{i-1})]\\
&\ge&\sum_{i\in I_s\cap J_g}\left(\frac{\log D_i}{2}+
\frac{(1+2\d')\log D_i}{2}-\log c\right)
+\sum_{i\notin I_s\cap J_g}\log D_i\\
&\ge&\log|A|+\sum_{i\in I_s\cap J_g}\d'\log D_i-n\log c.
\eean
The second inequality follows form Corollary \ref{cor_l}
and $c$ is the implied constant there.
And
$\AA\land\BB$ denotes the finest partition that is coarser
than both $\AA$ and $\BB$.
This implies in turn that for some choices of $x_1,\ldots,x_{2^m-1}$,
we have
\[
|A.x_1.A.x_2\ldots x_{2^m-1}.A.A|\ge c^{-n} |A||G|^{-\d}
\prod_{i\in I_s}D_i^{\d'},
\]
where we also used (\ref{eq_B2}).
Note that we can assume $c^n<|G|^{\d}$ by (\ref{eq_size}),
and recall that $|A|>|S||G|^{-2\d}$ by (\ref{eq_A}), hence
Proposition \ref{prp_sd} follows.
\end{proof}

\subsection{The case of many large degrees}
\label{sec_ld}

This section is devoted to the proof of
\begin{prp}
\label{prp_ld}
There is a positive constant $\d'$ depending only on
$\e$ and $L$, such that
\[
{\textstyle|\prod_{12} S|}\ge|G|^{\d'-\d}\prod_{i\in I_l}D_i
\]
\end{prp}

Recall that $G_s=\times_{i\in I_s}G_i$, $G_l=\times_{i\in I_l}G_i$
and $\pi_s$ and $\pi_l$ are the projections to these subgroups
respectively.

%
By (A3), any nontrivial representation of $G_i$ is of dimension
at least $|G_i|^{1/L}$.
It was pointed out by Nikolov and Pyber \cite[Corollary 1]{NP}
that a result of Gowers \cite[Theorem 3.3]{Gow} imply that if
$A,B,C\subset G_i$ are subsets such that $|A||B||C|>|G_i|^{3-1/L}$
then $A.B.C=G_i$.

Let $i_1\le\ldots\le i_{n'}$ be the indices in $I_l$ and for
$1\le n''\le n'$ set
$G_{\{i_1,\ldots,i_{n''}\}}=G_{i_1}\times\ldots\times G_{i_{n''}}$
and denote by $\pi_{\{i_1,\ldots,i_{n''}\}}$ the
projection to this subgroup.
We prove by induction that
\[
\pi_{\{i_1,\ldots,i_{n''}\}}(A.A.A)=G_{\{i_1,\ldots,i_{n''}\}}.
\]
For $n''=1$, this follows directly from \cite[Corollary 1]{NP} and
from $\pi_{i_1}(A)\ge D_{i_1}\ge |G_i|^{1-1/3L}$.
Now assume that the claim holds for some $n''$ and
take an arbitrary element $g\in G_{\{i_1,\ldots,i_{n''+1}\}}$.
By the induction hypothesis there are elements $h_1,h_2,h_3\in A$
such that
\[
\pi_{\{i_1,\ldots,i_{n''}\}}(h_1h_2h_3)
=\pi_{\{i_1,\ldots,i_{n''}\}}(g).
\]
Define the sets
\[
B_i=\{x\in A\sep\pi_{\{i_1,\ldots,i_{n''}\}}(x)
=\pi_{\{i_1,\ldots,i_{n''}\}}(h_i)\}
\]
and note that
\[
\pi_{i_{n''+1}}(B_i)\supset
\pi_{i_{n''+1}}(\{x\in A\sep\pi_{\le i_{n''+1}-1}(x)=
\pi_{\le i_{n''+1}-1}(h_i)\})
\]
hence
$|\pi_{i_{n''+1}}(B_i)|\ge D_{i_{n''+1}}
\ge|G_{i_{n''+1}}|^{1-1/3L}$.
Now an application of \cite[Corollary 1]{NP} to the sets
$\pi_{i_{n''+1}}(B_i)$ gives that
$g\in\pi_{\{i_1,\ldots,i_{n''+1}\}}(A.A.A)$ whence the claim follows.

Define the distance of two elements $g,h\in G_s$
by
\[
d(g,h)=\sum_{i\in I_s:\,\pi_i(g)\neq\pi_i(h)}\log|G_i|.
\]
\begin{lem}
\label{lem_a}
If $|S.S.S|\le|G|^{1-\e+\d}$ then there is an element $g\in\prod_9S$
such that
\[
\pi_l(g)=1\quad {\rm and}\quad d(\pi_s(g),1)>\d'\log|G|,
\]
where $\d'>0$ is a constant depending only on $\e$ and $L$.
\end{lem}

Following Farah \cite{Far}, we say that a map $\psi:G_l\to G_s$ is a
$\d'$--approximate homomorphism if
\[
d(\psi(g)\psi(h),\psi(gh))\le\d'\qquad{\rm and}
\]
\[
d(\psi(g),(\psi(g^{-1}))^{-1})\le\d'
\]
for all $g,h\in G_l$.
Note that in \cite{Far}, such a $\psi$ is called an approximate
homomorphism of type II.
We recall a result of Farah \cite[Theorem 2.1]{Far} that will be
crucial in the proof.
Let $\psi:G_l\to G_s$ be a $\d'$--approximate homomorphism.
Then there is a homomorphism $\f:G_l\to G_s$ such that
\[
d(\psi(g),\f(g))\le24\d'
\]
for all $g\in G_l$.

\begin{proof}[Proof of Lemma \ref{lem_a}]
Assume to the contrary that for any $g\in\prod_9 S$ with
$\pi_l(g)=1$, we have $d(\pi_s(g),1)\le\d'\log|G|$.
For each $g\in G_l$, pick an element $h\in S.S.S$ with $\pi_l(h)=g$
and set $\psi(g)=\pi_s(h)$.
This gives rise to a map $\psi :G_l\to G_s$, which of course
depends on our choices for $h$.
It follows in turn that for any $g\in G_l$ and $h\in S.S.S$
with $\pi_l(h)=g$, we have $d(\pi_s(h),\psi(g))<\d'\log|G|$
and that $\psi$ is a $\d'\log|G|$--approximate homomorphism.
By \cite[Theorem 2.1]{Far}, there is a homomorphism $\f$
with $d(\psi(g),\f(g))\le24\d'\log|G|$ for any $g\in G_l$.
The elements $g\in G$ satisfying
\[
\pi_s(g)=\f(\pi_l(g))
\]
constitutes a subgroup $H<G$ of index $|G_s|$,
since the cosets of $H$ are represented by the elements $g$
with $\pi_l(g)=1$.
For $h_1\in S.S.S$, the coset $h_1H$ is represented by the element $g_1$
with $\pi_l(g_1)=1$ and $\pi_s(g_1)=\pi_s(h_1)\f(\pi_l(h_1))^{-1}$.
Since
\bean
d(\pi_s(h_1),\f(\pi_l(h_1)))
&\le& d(\pi_s(h_1),\psi(\pi_l(h_1)))+d(\psi(\pi_l(h_1)),\f(\pi_l(h_1)))\\
&<&25\d'\log|G|,
\eean
there
is an index set
$I\subset I_s$ with
$\prod_{i\in I}|G_i|<|G|^{25\d'}$ such that $\pi_i(g_1)\neq1$
exactly if $i\in I$.
If $I$ is given there are at most $|G|^{25\d'}$ choices
for $g_1$.
Thus $S.S.S$ is contained in
$2^n|G|^{25\d'}<|G|^{26\d'}$ cosets of $H$.
This is a contradiction if
\[
|G_s|^{-\e}|G|^{26\d'+\d}<1.
\]
Since $|G_l|\le |S.S.S|\le|G|^{1-\e+\d}$,
we have
$|G_s|\ge|G|^{\e-\d}$.
Now, if $\d$ is small enough (e.g. $\d<\e^2/10$)
we can get the desired contradiction
by an appropriate choice of $\d'$.
\end{proof}

\begin{proof}[Proof of Proposition \ref{prp_ld}]
First we calculate the index of the centralizer $\CC(g)$
of $g$, the element constructed in Lemma \ref{lem_a}.
An element $h$ commutes with $g$ if and only if
$\pi_i(h)\in\CC(\pi_i(g))$ for all indices $i$ for which
$\pi_i(g)\neq1$.
For such an $i$, $[G_i:\CC(\pi_i(g))]>|G_i|^{1/L}$.
Recall that we assume that all the $G_i$ are simple, in particular
their centers are trivial.
Now we see that $[G:\CC(g)]>|G|^{\d'/L}$ with the $\d'$ of Lemma
\ref{lem_a}.
Then $S$ contains elements from at least $|G|^{\e\d'/L-\d}$
cosets of $\CC(g)$.
Thus the set
\[
\{sas^{-1}\sep s\in S\}\subset{\textstyle\prod_{11}S}
\]
contains at least $|G|^{\e\d'/L-\d}$ different elements $h$ with
$\pi_l(h)=1$, whence 
\[
|\prod_{12} S|\ge|G|^{\e\d'/L-\d}\prod_{i\in I_l}D_i,
\]
which was to be proven.
\end{proof}
We conclude with the
\begin{proof}[Proof of Proposition \ref{prp_prod}]
By Propositions \ref{prp_sd} and \ref{prp_ld}, we have
\[
|\textstyle\prod_{2^{m+1}} S|>|S||G|^{-Q\d}\prod_{i\in I_s}D_i^{\d_1'}
\qquad{\rm and}
\]
\[
|\textstyle\prod_{12}S|>|G|^{\d'_2-\d}\prod_{i\in I_l}D_i.
\]
with some constants $\d_1',\d_2'$ and $Q$.
Multiply the first inequality with the $\d_1'$th power of the second
one, and use $|G|\ge|S|$ and $\prod D_i=|A|\ge|S||G|^{-2\d}$ to get
\[
|\textstyle\prod_{2^m+1} S|
|\textstyle\prod_{12} S|^{\d_1'}>
|S|^{1+\d_1'+\d_1'\d_2'}|G|^{-Q'\d}.
\]
By the hypothesis on the set $S$ for $H=\{1\}$, we get $|S|>|G|^{\e-\d}$.
Therefore (\ref{eq_iterated}) gives the claim if $\d$ is sufficiently small.
\end{proof}

\section{(A1)--(A5) for $G_i=SL_d(\F_{p^k})$}
\label{sec_sld}

Let $K$ be a number-field and let $I\subset\OO_K$ be a square-free ideal.
Then $I=P_1\cdots P_n$ for some prime ideals, and
$G=SL_d(\OO_K/I)=SL_d(\OO_K/P_1)\times\cdots\times SL_d(\OO_K/P_n)$.
The last ingredient we need for the proof of Theorem \ref{thm_main}
is that the groups $G_i=SL_d(\OO_K/P_1)$ satisfy the assumptions
(A1)--(A5).
We write $\F_{p^k}$ for the finite field of order $p^k$.

(A1) is immediate, and (A2) is a classical result of Jordan.
Regarding (A3),
Harris an Hering \cite{HaHe} proved that any nontrivial representation
of $SL_d(\F_q)$ is of dimension at least $q^{d-1}-1$ or $(q-1)/2$
when $d=2$ and $q$ is odd.
In fact for our purposes it is enough to note that any such representation
restricted to an appropriate subgroup isomorphic to $SL_2(\F_p)$
gives rise to a nontrivial representation, which is of dimension at
least $(p-1)/2$ by a classical result of Frobenius \cite{Fro}.

We study (A4) and (A5) in the next two sections.

\subsection{Assumption (A4)}
\label{sec_a4}
We recall some results of Helfgott.
Let $G=SL_d(\F_{p})$, and
let $S\subset G$ be a set which is not contained in any
proper subgroup.
Suppose further that $|S|<|G|^{1-\e}$ for some $\e>0$.
Then if $d=2$ \cite[Key Proposition]{Hel}
or if $d=3$ \cite[Main Theorem]{He2}, there is a $\d>0$ depending only on $\e$
such that $|S.S.S|\gg|S|^{1+\d}$.
These results imply (A4) for $G_i=SL_d(\F_{p_i})$
if $d=2$ or $d=3$ the same way
as we proved Theorem \ref{thm_prod} using Proposition \ref{prp_prod}.
We show below that the argument in \cite{Hel} extends easily for groups
$G=SL_2(\F_{p^k})$.
After the circulation of an early version of this paper
I have learnt that this extension of Helfgott's theorem was recently
proven by Oren Dinai in his PhD thesis \cite{Din}.

Let $\L$ be a subset of the multiplicative group
$\F_{p^k}^*$.
Denote by $\L^r$ the set of $r$'th powers of the elements of $\L$
and set
\[
w(\L)=\{w(a)\sep a \in\L\},\qquad{\rm where}\qquad
w(a)=a+a^{-1}.
\]
The only notable change needed to extend Helfgott's argument
for the case $k>1$ is to replace  \cite[Proposition 3.3]{Hel}
by the following
\begin{prp}
\label{prp_Hel}
Let $\L\subset \F_{p^k}^*$ be a set which contains $1$ and is
closed under taking multiplicative inverses.
Let $a_1,a_2\in\F_{p^k}^*$, and assume that if $w(\L^2)$
is contained in a proper subfield $F$ of $\F_{p^k}$,
then $a_1/a_2\notin F$.
Now if $|\L|<p^{(1-\d)k}$, then
\[
\left|\{a_1w(bc)+a_2w(bc^{-1})\sep b,c\in\textstyle\prod_4\L\}\right|
\gg|\L|^{1+\e}
\]
with a constant $\e$ depending only on $\d$.
\end{prp}
The proof follows the same lines as that of
\cite[Proposition 3.3]{Hel}.
\begin{proof}
Set $\L_1=\L^2.\L^2$.
Using the substitution $b=\bar b\bar c$ and $c=\bar b{\bar c}^{-1}$,
we see that
\bean
a_1w(\L_1)+a_2w(\L_1)
&=&\{a_1w({\bar b}^2)+a_2w({\bar c}^2)\sep\bar b, \bar c\in\L.\L\}\\
&\subset&\{a_1w(bc)+a_2w(bc^{-1})
\sep b, c\in\textstyle\prod_4\L\}.
\eean
If $w(\L^2)$ is contained in a subfield $F$, then $a_1/a_2\notin F$
by assumption, and then trivially
\[
|a_1w(\L_1)+a_2w(\L_1)|\ge|w(\L^2)|^2\ge\frac{1}{16}|\L|^2,
\]
and the claim follows.

Therefore we will assume now that $w(\L^2)$ generates $\F_{p^k}$.
Assume that
\be
\label{eq_a4}
|a_1/a_2 w(\L_1)+w(\L_1)|\le K|\L|
\ee
for some constant $K$.
By the Ruzsa-Pl\"unnecke inequalities \cite{Ru2}
(see also \cite[Corollary 6.9]{TaV}) 
\[
|w(\L_1)+w(\L_1)-w(\L_1)-w(\L_1)|\ll K^4|\L|.
\]
Note that $w(a)w(b)=w(ab)+w(ab^{-1})$, hence
\[
w(\L^2).w(\L^2)\subset w(\L_1)+w(\L_1)
\]
and
\[
|w(\L^2).w(\L^2)-w(\L^2).w(\L^2)|\ll K^4|\L|.
\]
This would contradict the sum-product theorem if
$K=|\L|^\e$ with $\e$ small enough.
The most convenient reference for us is \cite[Theorem 1.5]{Ta2}
that we can apply
with $A=w(\L^2)$ and $a=w(1)=2$.
However the contradiction could also be decuded from the results of
\cite{BKT} or \cite{BGK}.
\end{proof}

To use this proposition we need to replace
\cite[Corollary 4.5]{Hel} by
\begin{lem}
\label{lem_Hel2}
Let $S\subset SL_2(\F_{p^k})$ be  symmetric 
containing 1, and assume that it is not
contained in any proper subgroup.
Let $F$ be a proper subfield of $\F_{p^k}$.
Then there is an absolute constant $R$ such that there is a matrix
\[x=\mat abcd\in\textstyle\prod_R S\]
with $abcd\neq0$ and $ad/bc\notin F$.
\end{lem}
\begin{proof}
In this proof the value of $R$ may be different at different
occurrences.
First note that for any matrix $x$ with entries as above,
$ad+bc=1\in F$ and hence
\[
bc=\frac{1}{ad/bc-1},
\]
so $x$  satisfy the requirements
of the lemma exactly if $bc\notin F$.
If $x$ does not satisfy this, look at $x^2$ and notice that
the product of the off-diagonal entries is $bc(\Tr x)^2$,
hence it remains to show that $\prod_{N} S$ contains an element
with nonzero off-diagonals and with $(\Tr x)^2\notin F$.

Note that if ${\rm span}(\prod_l S)={\rm span}(\prod_{l+1}S)$,
where ${\rm span} (X)$ denotes $\F_{p^k}$--linear span in
$Mat_2(\F_{p^k})$,
then ${\rm span}(\prod_m S)={\rm span}(\prod_l S)$ for any $m>l$.
From this we conclude that as $S$ is not contained in a proper
subgroup, $\prod_4 S$ must span $Mat_2(\F_{p^k})$.
Let $y_1,y_2,y_3,y_4\in \prod_4 S$ be a basis of $Mat_2(\F_{p^k})$ and
let $z_1,z_2,z_3,z_4$ be the dual basis with respect to the
non-degenerate form $\Tr(yz)$.
Denote by $\o$ an element of $\F_{p^k}$ which is not in $F$
but $\o^2\in F$.
If there is no such element, the rest of the proof is even simpler.
Consider the 16 $F$--vectorspaces
\[
\omega^{\a_1}F\cdot z_1+\omega^{\a_2}F\cdot z_2
+\omega^{\a_3}F\cdot z_1+\omega^{\a_4}F\cdot z_1,
\]
where the $\a_i$ takes the values 0 and 1 independently.
Now we invoke Lemma 4.4 from Helfgott \cite{Hel}, which
gives that there is a matrix $\bar x\in\prod_{R}S$
which is not contained in any of the above subspaces
if $R$ is large enough.
By definition, there is an index $i$ such that
$(\Tr(y_i\bar x))^2\notin F$.
It may happen that one or both off-diagonal entries are zero.
Using \cite[Lemma 4.4]{Hel} now for the representation of
$SL_2(\F_{p^k})$ acting on $Mat_2(\F_{p^k})$ by conjugations, we see
that $wy_i\bar xw^{-1}$ has no zero entries for some $w\in\prod_R S$.
This proves the claim.
\end{proof}

We remark, that in the way \cite[Lemma 4.4]{Hel} is stated,
it gives an $R$ which depends on the dimension of $Mat(\F_{p^k})$
over $F$, however it is easily seen by a careful analysis of the
proof in \cite{Hel} that the dependence is only on the dimension
of the subspaces we want to avoid.

\begin{proof}[Extending {\cite[Key Proposition]{Hel}} to
arbitrary finite fields]
The proof on pp. 616 \cite{Hel} is given for arbitrary
finite fields up to the point when the set $V$ is constructed,
except that we get $|V|<p^{k(1-\d/3)}$ not
$|V|<p^{1-\d/3}$.
If $w(V)$ is contained in a proper subfield of $\F_{p^k}$ then denote
this subfield by $F$, and instead \cite[Corollary 4.5]{Hel}
use Lemma \ref{lem_Hel2} to construct the matrix
\[
\mat abcd.
\]
In what follows simply use Proposition \ref{prp_Hel} instead of
\cite[Proposition 3.3]{Hel}. 
\end{proof}

\subsection{Assumption (A5)}
\label{sec_nor}

We prove that $SL_d(\F_{p^k})$ satisfies (A5) with
$L$ depending on $d$ and $k$.
Note that we can embed $SL_d(\F_{p^k})$ into $GL_{kd}(\F_p)$
by Weil restriction.
We again rely on the description of the subgroup structure
of $GL_d(\F_p)$ given by Nori \cite{Nor}.
Recall that for a group $H<GL_d(\F_p)$, $H^+$ denotes
the subgroup generated by elements of order $p$.
By \cite[Theorem B]{Nor} there is a connected algebraic
subgroup $\wt H<GL_d$ such that $\wt H(\F_p)^+=H^+$.
By \cite[Theorem C]{Nor} there is a commutative $F<H$
such that $p\nmid |F|$ and $H\la_{L_1} FH$  with a constant $L_1$
depending only on $d$.
Moreover, it follows from the proof there, that if $P$ is any
$p$--Sylow subgroup of $H^+$, then $F$ can be chosen to satisfy
\be
\label{eq_F}
F<\NN_H(P),\quad F\cap P=\emptyset\quad{\rm and}\quad
[\NN_H(P):FP]<L_1.
\ee
The choice of $F$ is not unique, even for a fixed
Sylow subgroup $P$, however the following is true.
Let $K<\NN_H(P)/P$ be a group whose order is prime to $P$.
Then there is an $F<\NN_H(P)$ with $K=FP/P$ by
\cite[Theorem 7.41]{Rot} and all such subgroups $F$ are conjugates
of each other by elements of $P$, see Rotman
\cite[Theorem 7.42]{Rot}.

\begin{prp}
\label{prp_a5}
Let $G$ be a quasi-simple subgroup of $GL_d(\F_p)$
such that $G=G^+$.
There are classes $\HH_0,\ldots,\HH_m$ of subgroups of
$G$ such that the following hold with some constants
$L, m$ depending only on $d$:
\begin{itemize}
\item[$(i)$]
$\HH_0=\{Z(G)\}$,
\item[$(ii)$]
each $\HH_i$ is closed under conjugation by elements of $G$.
\item[$(iii)$]
for every proper subgroup $H<G$ there is some $i$ and
a subgroup $H^\sharp\in\HH_i$ such that $H\la_L H^\sharp$,
\item[$(iv)$]
for every pair of subgroups $H_1,H_2\in\HH_i,\;H_1\neq H_2$
there is some
$i'<i$ and $H^\sharp\in\HH_{i'}$ such that
$H_1\cap H_2\la_L H^\sharp$,

\end{itemize}
\end{prp}

\begin{proof}
In each subgroup $H<G$ which is generated by elements of order $p$,
distinguish a $p$-Sylow subgroup $P$.
This can be arbitrary, but should be fixed throughout the proof.
For integers $i$ and $j$ we define the classes $\HH_{i,j}$.
A subgroup $H<G$ belongs to $\HH_{i,j}$ precisely if $Z(G)<H$, $\dim\wt H=i$
and $j$ is the
least integer for which the following hold.
There is a commutative subgroup $F<\NN_H(P)$ such that
\be
\label{eq_Hij1}
\begin{split}
Z(G)<F,\quad
H=FH^+,\quad
F\cap P=\emptyset\quad {\rm and}\\
[\NN_{H^+}(P):(F\cap H^{+})P]<L_1^{2^{d-j}},
\end{split}
\ee
and there
is a $j$ dimensional subspace $V$ of $Mat_d(\F_p)$ such that
\be
\label{eq_Hij2}
F=V\cap\NN_G(P)\cap\NN_G(H^+).
\ee
Order the nonempty classes $\HH_{i,j}$ in such a way that
$\HH_{i,j}$ preceeds $\HH_{i',j'}$ if $i<i'$ or $i=i'$
and $j<j'$.

The first nonempty class is
$\HH_{0,j}=\{Z(G)\}$ for some $j$, and $(i)$ follows.
Since conjugation is a linear transformation on $Mat_d(\F_p)$,
$(ii)$ is clear.

Let $H<G$ be a proper subgroup, and replace it by $Z(G)H$
if necessary, to ensure that $Z(G)<H$.
Let $F$ be a subgroup of $\NN_H(P)$ that satisfies
(\ref{eq_F}).
Without loss of generality, we can assume that $Z(G)<F$.
Set
\[
F^\sharp={\rm span}(F)\cap\NN_G(P)\cap\NN_G(H^+),
\]
where ${\rm span}(F)$ is the linear span of $F$ in the vectorspace
$Mat_d(\F_p)$.
First we remark that $F^{\sharp}$ does not contain an element
of order $p$, in fact its elements can be mutually diagonalized
over an appropriate extension field.
This implies that $F^{\sharp}\cap P=\emptyset$.
Since $F^{\sharp}\subset\NN_G(H^+)$, we can define the subgroup
$H^{\sharp}=F^{\sharp}H^+$, and we have $(H^{\sharp})^+=H^+$.
Since $[H:FH^+]<L_1$ and $FH^+<H^\sharp$,
for $(iii)$ we only need to show that $H^\sharp\in\HH_{i,j}$
for some $i$ and $j$.
This holds with $i=\dim \wt H$ and with some $j\le\dim{\rm span}(F)$,
since
$F^\sharp$ is commutative, and we have
\[
[\NN_{H^+}(P):(F^\sharp\cap H^+)P]\le[\NN_{H^+}(P):(F\cap H^+)P]=
[\NN_{FH^+}(P):FP]\le L_1.
\]
Here the equation in the middle
follows from the fact $\NN_{FH^+}(P)=F\NN_{H^+}(P)$.

It remains to show $(iv)$.
Let $H_1$ and $H_2$ be two different groups in $\HH_{i,j}$.
If $\wt H_1\neq \wt H_2$, then
\[
\dim (\wt H_1\cap\wt H_2)\le\dim \wt{H_1\cap H_2}<i
\]
and $(H_1\cap H_2)^\sharp\in\HH_{i',j'}$ with some
$i'<i$.
Therefore we may assume $\wt H_1=\wt H_2$ and hence
$H_1^+=H_2^+$.
Let $P$ be the distinguished $p$-Sylow subgroup
and denote by $F_l<\NN_{H_l}(P)$ and $V_l$ ($l=1,2$) the subgroups and
subspaces for which $(\ref{eq_Hij1})$ and $(\ref{eq_Hij2})$ hold.
We show that there is an $H\in\HH_{i,j'}$ for some $j'<j$ such that
$H_1\cap H_2\la_{L_1^{2^{d-j+1}}}H$.
We have $[\NN_{H_l}(P):F_lP]<L_1^{2^{d-j}}$ for $l=1,2$, hence
\[
[\NN_{H_1\cap H_2}(P):F_1P\cap F_2P]<L_1^{2^{d-j+1}}.
\]
By \cite[Theorem 7.41]{Rot}
as mentioned before, there is a subgroup $F<\NN_H(P)$ with $F\cap P=\emptyset$
and $FP=F_1P\cap F_2P$.
Moreover, since conjugation is linear
we can assume by \cite[Theorem 7.42]{Rot}  that $F=F_1\cap F_2$.
The claim follows if we define $H=FH^+$, since
\[
[\NN_{H^+}(P):(F\cap H^+)P]\le[\NN_{H^+}(P):(F_1\cap H^+)P]\cdot
[\NN_{H^+}(P):(F_2\cap H^+)P]
\]
and $\dim (V_1\cap V_2)<j$.

\end{proof}

\section{Proof of Theorem \ref{thm_main}}
\label{sec_proof}

Let notation be the same as in the statement of the theorem.
First we note that by \cite[Claim 11.19]{HLW},
it is enough to prove that $\GG(SL_d(\OO_K/I),\pi_I(S'))$
form a family of expanders with some $S'\subset\Ga$, hence
we can assume without loss of generality that Theorem \ref{thm_escp}
holds with $S=S'$.
If $H<SL_d(\OO_K/I)$ and $\pi_I(S)\subset H$, then by
Theorem \ref{thm_escp}, $[SL_d(\OO_K/I):H]<C$ for some constant $C$
which depends on the $\d$ and the implied constant of that theorem.
Let $J$ be a square-free ideal for whose prime factors $P$,
$\pi_P(S)$ does not generate $SL_d(\OO_K/P)$.
Since each proper subgroup in $SL_d(\OO_K/P)$ is of index at least
$N(P)^{\d'}$ for some $\d'>0$, we get $N(J)<C^{\d'}$.
Here, and everywhere below $\d'$ is a constant which may depend on $S$
and which need not be the same at different occurrences.
Thus there is at most a finite number of prime ideals $P$
such that $\pi_P(S)$ is not generating, and from now on, we denote
by $J$ the product of those prime ideals.

Let $I$ be an ideal which is prime to $J$ and write $G=SL_d(\OO_K/I)$,
and $\overline S=\pi_I(S)$.
Denote by $l^2(G)$ the vectorspace of complex valued functions
on $G$.
Consider the operator on $l^2(G)$ which is convolution by
$\chi_{\overline S}$ from the left.
Denote its matrix in the standard basis by $M$.
It is plain that $|S|M$ is the adjacency matrix of the graph
$\GG(G,\overline S)$.
In light of the results of Dodziuk \cite{Dod}, Alon and Milman
\cite{AlMi} and Alon \cite{Alo} already mentioned in the introduction,
we have to give an upper bound on the second largest eigenvalue
of $M$ independently of $I$.
For $g\in G$, denote by $\a(g)$ the left translation by $g$ on $l^2(G)$.
$\a$ is called the regular representation of $G$, and it is well known
that $l^2(G)$ decomposes as a direct sum
$V_0\oplus V_1\oplus\ldots\oplus V_m$, such that each $\a\restr{V_i}$
is irreducible and the multiplicity of every irreducible representation
of $G$ in this decomposition is the same as its dimension.
Therefore it is left to show that if $\b$ is a nontrivial
irreducible representation
of $G$, and $\l$ is an eigenvalue of the operator
\[
\frac{1}{|S|}\sum_{g\in\overline S}\beta(s),
\]
then $\l<c<1$ for some constant $c$ independent of $I$.
Replacing $I$ by a larger ideal if necessary, we may assume that the
representation is faithful.
Faithful representations of $G$ are tensor products of nontrivial
representations of the direct factors, hence they are of dimension
at least $|G|^{\d'}$ as we noted at the beginning of section \ref{sec_sld}.
Hence $\l$ is an eigenvalue of $M$ with multiplicity at least $|G|^{\d'}$.

Denote by $(M)_{i,j}$ the $i,j$ entry of $M$ and notice
that for an
integer $k$, the rows of $M^k$ are translates of $\chi_{\overline S}^{(k)}$.
Then
\[
\Tr(M^{2k})=\sum_{i,j\le|G|}(M^k)_{i,j}^2=|G|\|\chi_{\overline S}^{(k)}\|_2^2,
\]
whence
\be
\label{eq_eigen}
\l^{2k}\le|G|^{1-\d'}\|\chi_{\overline S}^{k}\|_2^2.
\ee
If the index of a subgroup $H<SL_d(\OO_K/I)$ is large, then we can
cancel the implied constant in Theorem \ref{thm_escp} by making $\d$
smaller.
If the index is small, then we can get a nontrivial bound
$\chi_{\overline S}^{(k)}(H)<c<1$, since we assumed that $\overline S$
generates the group.
Thus if $I$ is restricted to ideals prime to $J$, Theorem \ref{thm_escp}
holds with the implied constant set to 1.
Now apply it for $H=\{1\}$ to get
\[
\|\chi_{\overline S}^{(\log N(I))}\|_2^2<|G|^{-\d'}.
\]
We saw in section \ref{sec_sld} that $G$ satisfies (A0)--(A3) and (A5).
It also satisfy (A4) if $d=2$ or if $d=3$ and $K=\Q$ or if we assume that
(\ref{eq_helfgott}) holds if $F$ ranges over the fields $\OO_K/P$,
$P$ prime.
Therefore we can apply Theorem \ref{thm_prod} repeatedly
to get
\[
\|\chi_{\overline S}^{(C\log(N(I)))}\|_2^2<|G|^{-1+\e}
\]
for arbitrary $\e>0$ with some constant $C$ depending on $\e$.
If $\e$ is less than the $\d'$ in (\ref{eq_eigen}), the theorem follows.


\bigskip

{\sc Department of Mathematics, Princeton University, Princeton,
NJ 08544, USA and

Analysis and Stochastics Research Group of the
Hungarian Academy of Sciences, University of Szeged, Szeged, Hungary}

{\em e-mail address:} pvarju@princeton.edu
\end{document}